\documentclass[11pt]{article}
\usepackage[margin=1in]{geometry} 
\usepackage[cmex10]{amsmath} 
\usepackage{amsthm}
\ifx\UseOption\undefined
\def\UseOption{tr}
\fi
\usepackage{optional}

\usepackage{amssymb}
\usepackage{algorithm}
\usepackage{algorithmic}
\usepackage{color}
\usepackage{mathtools}
\usepackage{enumitem} 
\usepackage{hyperref}
\theoremstyle{plain}
\newtheorem{theorem}{Theorem}
\newtheorem{proposition}{Proposition}
\newtheorem{lemma}{Lemma}
\newtheorem{corollary}{Corollary}
\newtheorem{claim}{Claim}

\theoremstyle{definition}

\theoremstyle{remark}
\newtheorem{remark}{Remark}

\def\ba#1\ea{\begin{align*}#1\end{align*}}
\def\ban#1\ean{\begin{align}#1\end{align}}

\allowdisplaybreaks 
\def\argmax{\operatornamewithlimits{arg\,max}} 
\newcommand{\argmin}{\mathop{\mathrm{arg\min}}}

\newcommand{\bE}{\ensuremath{\mathbb{E}} }
\newcommand{\bP}{\ensuremath{\mathbb{P}} }
\newcommand{\bR}{\ensuremath{\mathbb{R}} }
\newcommand{\bZ}{\ensuremath{\mathbb{Z}} }
\newcommand{\bEq}[1]{ \mathbb{E} \left[\left. #1\right|\bf{q}(t)=\bf{q}\right]  }

\newcommand{\vq}{\ensuremath{{\bf q}} }
\newcommand{\bvq}{\ensuremath{\overline{{\bf q}}} } 
\newcommand{\va}{\ensuremath{{\bf a}} }
\newcommand{\bva}{\ensuremath{\overline{{\bf a}}} } 
\newcommand{\vs}{\ensuremath{{\bf s}} }
\newcommand{\vu}{\ensuremath{{\bf u}} }
\newcommand{\vlam}{\ensuremath{{\mbox{\boldmath{$\lambda$}}}} } 
\newcommand{\vsig}{\ensuremath{{\mbox{\boldmath{$\sigma$}}}} } 
\newcommand{\vnu}{\ensuremath{{\mbox{\boldmath{$\nu$}}}} } 
\newcommand{\vqt}{\ensuremath{{\bf q}(t)} }

\newcommand{\vx}{\ensuremath{{\bf x}} }
\newcommand{\vy}{\ensuremath{{\bf y}} }

\newcommand{\vei}{\ensuremath{{\bf e}^{(i)}} }
\newcommand{\vetj}{\ensuremath{\widetilde{{\bf e}}^{(j)}} }
\newcommand{\ei}{\ensuremath{e^{(i)}} }
\newcommand{\etj}{\ensuremath{\widetilde{ e}^{(j)}} }

\newcommand{\vone}{\ensuremath{{\bf 1}} } 

\newcommand{\qij}{\ensuremath{q_{ij}} }
\newcommand{\bqij}{\ensuremath{\overline{q}_{ij}} } 
\newcommand{\qijt}{\ensuremath{q_{ij}(t)} }
\newcommand{\bq}{\ensuremath{\overline{q}} }
\newcommand{\aij}{\ensuremath{a_{ij}} }
\newcommand{\baij}{\ensuremath{\overline{a}_{ij}} } 
\newcommand{\sij}{\ensuremath{s_{ij}} }
\newcommand{\uij}{\ensuremath{u_{ij}} }
\newcommand{\lij}{\ensuremath{_{ij}} }
\newcommand{\para}{\ensuremath{_{\parallel}} }
\newcommand{\per}{\ensuremath{_{\perp}} }
\newcommand{\paraij}{\ensuremath{_{\parallel ij}} }
\newcommand{\perij}{\ensuremath{_{\perp ij}} }
\newcommand{\se}{\ensuremath{^{(\epsilon)}} }
\newcommand{\numin}{\ensuremath{\nu_{\min}} }
\newcommand{\rt}{\ensuremath{\widetilde{r}} }

\newcommand{\cC}{\ensuremath{\mathcal{C}} }
\newcommand{\cS}{\ensuremath{\mathcal{S}} }
\newcommand{\cK}{\ensuremath{\mathcal{K}} }
\newcommand{\cF}{\ensuremath{\mathcal{F}} }
\newcommand{\cKo}{\ensuremath{\mathcal{K}^{\circ}} }

\newcommand{\La}{\left \langle }
\newcommand{\Ra}{\right \rangle}

\newcommand{\V}{\ensuremath{V} }
\newcommand{\Vq}{\ensuremath{V(\vq)} }
\newcommand{\Vi}{\ensuremath{V_1} }
\newcommand{\Viq}{\ensuremath{V_1(\vq)} }
\newcommand{\Vo}{\ensuremath{V_2} }
\newcommand{\Voq}{\ensuremath{V_2(\vq)} }
\newcommand{\Vt}{\ensuremath{V_3} }
\newcommand{\Vtq}{\ensuremath{V_3(\vq)} }
\newcommand{\W}{\ensuremath{W} }
\newcommand{\Wq}{\ensuremath{W(\vq)} }
\newcommand{\Wper}{\ensuremath{W_{\perp}} }
\newcommand{\Wperq}{\ensuremath{W_{\perp}(\vq)} }
\newcommand{\Vper}{\ensuremath{V_{\perp}} }
\newcommand{\Vperq}{\ensuremath{V_{\perp}(\vq)} }
\newcommand{\Vpara}{\ensuremath{V_{\parallel}} }
\newcommand{\Vparaq}{\ensuremath{V_{\parallel}(\vq)} }
\newcommand{\Vu}{\ensuremath{V_4} }
\newcommand{\Vuq}{\ensuremath{V_4(\vq)} }  

\newcommand{\Bi}{B_1}
\newcommand{\Bii}{B_2}
\newcommand{\Biii}{B_3}
\newcommand{\Biv}{B_4}

\newcommand{\Ti}{\ensuremath{\mathcal{T}_1} }
\newcommand{\Tii}{\ensuremath{\mathcal{T}_2} }
\newcommand{\Tiii}{\ensuremath{\mathcal{T}_3} }
\newcommand{\Tiv}{\ensuremath{\mathcal{T}_4} }

\begin{document}
\title{Queue Length Behavior in a Switch under the MaxWeight Algorithm }
\author{Siva Theja Maguluri\\Mathematical Sciences and Analytics 
\\IBM T. J. Watson Research Center\\ Yorktown Heights, NY 10598\\smagulu@us.ibm.com \and  R. Srikant\\
Department of ECE and CSL\\University of Illinois at Urbana-Champaign\\Urbana, IL 61801\\rsrikant@illinois.edu}

\date{}
\maketitle

\begin{abstract}
We consider a switch operating under the MaxWeight scheduling algorithm, under any traffic pattern such that all the ports are loaded. This system is interesting to study since the queue lengths exhibit a multi-dimensional state-space collapse in the heavy-traffic regime.  We use a Lyapunov-type drift technique to characterize the heavy-traffic behavior of the expectation of the sum queue lengths in steady-state, under the assumption that all ports are saturated and all queues receive non-zero traffic. Under these conditions, we show that the heavy-traffic scaled queue length is given by $\left(1-\frac{1}{2n}\right)||\sigma||^2$, where $\sigma$ is the vector of the standard deviations of arrivals to each port in the heavy-traffic limit. In the special case of uniform Bernoulli arrivals, the corresponding formula is given by $\left(n-\frac{3}{2}+\frac{1}{2n}\right)$. The result shows that the heavy-traffic scaled queue length has optimal scaling with respect to $n,$ thus settling one version of an open conjecture; in fact, it is shown that the heavy-traffic queue length is at most within a factor of two from the optimal. We then consider certain asymptotic regimes where the load of the system scales simultaneously with the number of ports. We show that the MaxWeight algorithm has optimal queue length scaling behavior provided that the arrival rate approaches capacity sufficiently fast.
\end{abstract}


\section{Introduction}
Consider a collection of queues arranged in the form of an $n\times n$ matrix.  The queues are assumed to operate in discrete-time and jobs arriving to the queues will be called packets. The following constraints are imposed on the service process of the queueing system: (a) at most one queue can be served in each time slot in each row of the matrix, (b) at most one queue can be served in each time slot in each column of the matrix, and (c) when a queue is served, at most one packet can be removed from the queue. Such a queueing system is called a \emph{switch}.

A scheduling algorithm for the switch is a rule which selects the queues to be served in each time slot. A well-known algorithm called the MaxWeight algorithm is known to optimize the throughput in a switch. The algorithm was derived in a more general context in \cite{TasEph_92} and for the special context of the switch considered in here in \cite{mckeown96walrand}, where it was also shown that other seemingly good policies are not throughput-optimal. An important open question that is not fully understood is whether the MaxWeight algorithm is also queue length or delay optimal in any sense. In \cite{stolyar2004maxweight}, it was shown that the MaxWeight algorithm minimizes the sum of the squares of the queue lengths in heavy-traffic under a condition called \emph{Complete Resource Pooling} (CRP). For the switch, the CRP condition means that the arriving traffic saturates at most one column or one row of the switch. The result relies on the fact that, under CRP and in the heavy-traffic regime, there is a one-dimensional state-space collapse, i.e., the state of the system collapses to a line. When the CRP condition is not met, the state-space collapses to a lower-dimension, but is not one-dimensional. State-space collapse without the CRP condition was established in \cite{Stolyar_cone_SSC} when the arrivals are deterministic. For stochastic arrivals, state-space collapse for the fluid limit was studied in
 \cite{ShaWis_11}, and a diffusion limit has been established in \cite{kang2009diffusion}. However, a characterization of the steady-state behavior of the diffusion limit was still open.

In this paper, we use the Lyapunov-type drift technique introduced in \cite{erysri-heavytraffic}. The basic idea is to set the drift of an appropriately chosen function equal to zero in steady-state to obtain both upper and lower bounds on quantities of interest, such as the moments of the queue lengths. To obtain upper bounds one has to establish state-space collapse in a sense that is somewhat different than the one in \cite{stolyar2004maxweight}: the main difference being that the state-space collapse is expressed in terms of the moments of the queue lengths in steady-state. This form of state-space collapse can then be readily used in the drift condition to obtain the upper bound. However, in \cite{erysri-heavytraffic}, the usefulness of the drift technique was only established under the CRP condition. In this paper, we consider the switch with uniform traffic, i.e., where the arrival rates to all queues are equal. Thus, in the heavy-traffic regime, when the traffic in one column (or row) approaches its capacity, the traffic in all rows and columns approach capacity, and the CRP condition is violated. The main contribution of the paper is to characterize the expected steady-state queue lengths in heavy-traffic even though the CRP condition is violated. As mentioned earlier, when the CRP condition is violated, the state does not typically collapse to a single dimension. The main challenge in our proof is due to the difficulty in characterizing the behavior of the queue length process under such a multi-dimensional state-space collapse. Characterizing the behavior of the queue lengths under multi-dimensional state-space collapse has been difficult, in general, except in rare cases; see \cite{kang2009state, ji2010optimal} for two such examples in other contexts.

The difficulty in understanding the steady-state queue length behavior of the MaxWeight algorithm has meant that it is unknown whether the the MaxWeight algorithm minimizes the expected total queue length in steady-state. One way to pose the optimality question is to increase the number of queues in the system, or increase the arrival to a point close to the boundary of the capacity region (the heavy-traffic regime), or do both, and study whether the MaxWeight algorithm is queue-length-optimal in a scaling sense. A conjecture regarding the scaling behavior for any algorithm, both in heavy-traffic and under all traffic conditions, has been stated in \cite{shah_switch_open}. The authors first heard about the non-heavy-traffic version of this conjecture from A. L. Stolyar in 2005. The conjecture seemed to be difficult to verify for the MaxWeight algorithm, and so a number of other algorithms have been developed to achieve either optimal or near-optimal scaling behavior; see \cite{neely2007logarithmic,shah2012optimal,zhong2014scaling}. The results in this paper establish the validity of one version of the conjecture (pertaining to uniform traffic in the heavy-traffic regime) for the MaxWeight algorithm.

\emph{Note on Notation:}
The set of real numbers, and the set of non-negative real numbers
 are denoted by $\mathbb{R}$, and $\mathbb{R}_+$
  respectively.
We work in the $n^2-\text{dimensional}$ Euclidean space $\mathbb{R}^{n^2}$. We represent vectors in this space in bold font, by $\vx$. We use two indices $1\leq i \leq n$ and $1\leq j \leq n$ for different components of $\vx$. We represent the $(i,j)^{\text{th}}$ component by $x\lij$ and thus, $\vx = (x\lij)\lij$.
For two vectors $\vx$ and $\vy$ in $\mathbb{R}^{n^2}$, their inner product $\left\langle \vx,\vy\right\rangle $ and Euclidean norm $\|\vx\|$are defined by
\ba
\left\langle \vx,\vy\right\rangle  \triangleq \sum_{i=1}^n\sum_{j=1}^n x\lij y\lij,
    \quad \|\vx\|\triangleq \sqrt{\left\langle \vx,\vx\right\rangle} = \sqrt{\sum_{i=1}^n\sum_{j=1}^n x\lij^2}.
\ea
For two vectors $\vx$ and $\vy$ in $\mathbb{R}^{n^2}$, $\vx \leq \vy$ means $x\lij \leq y\lij$ for every $(i,j)$.
We use \vone to denote the all ones vector.
Let \vei denote the vector defined by $\ei\lij =1$ for all $j$ and $\ei_{i',j}=0$ for all $i' \neq i$ and for all $j$. Thus, \vei is a matrix with $i^{\text{th}}$ row being all ones and zeros every where else.
Similarly, let \vetj denote the vector defined by $\etj\lij =1$ for all $i$ and $\etj_{i,j'}=0$ for all $j' \neq j$ and for all $i$, i.e., it is a matrix with $j^{\text{th}}$ column being all ones and zeros every where else.
For a random process $\vq(t)$ and a Lyapunov function $V(.)$, we will sometimes use $V(t)$ to denote $V(\vq(t))$. We use $\text{Var}(.)$ to denote variance of a random variable.

\section{Preliminaries}
In this section, we will present the model of an input queued switch, MaxWeight scheduling algorithm, some observations on the geometry of the capacity region and other preliminaries.
\subsection{System Model and MaxWeight Algorithm}\label{sub:model}

An input queued switch is a model for cross-bar switches that are widely used. An $n\times n$ switch has $n$ input ports and $n$ output ports. We consider a discrete time system. In each time slot $t$, packets arrive at any of the input ports to be delivered to any of the output ports.  When scheduled, each packet needs one time slot to be transmitted across.

Each input port maintains $n$ separate queues, one each for packets to be delivered to each of the $n$ output ports.
We denote the queue length of packets at input port $i$ to be delivered at output port $j$ at time $t$ by $\qijt$. Let $\vq\in \bR^{n^2}$ denote the vector of all queue lengths.

Let $\aij(t)$ denote the number of packet arrivals at input port $i$ at time $t$ to be delivered to output port $j$, and we let $\va\in \bR^{n^2}$ denote the vector $(a\lij)\lij$. For every input-output pair (i,j), the arrival process $\aij(t)$ is a stochastic process that is i.i.d across time, with mean $\bE[\aij(t)]=\lambda \lij$ and variance $\text{Var}(\aij(t))=\sigma^2 \lij$ for any time $t$.
We assume that the arrival processes are independent across input-output pairs, (i.e, if $(i,j)\neq (i',j')$, the processes $a\lij(t)$ and $a_{i'j'}(t)$ are independent) and are also independent of the queue lengths or schedules chosen in the switch. We further assume that for all $i,j,t$, $\aij(t) \leq a_{\max}$  for some  $a_{\max}\geq 1$ and $P(\aij(t)=0)>\epsilon_a$ for some $\epsilon_a>0$.
The arrival rate vector is denoted by $\vlam = (\lambda \lij)\lij$ and the variance vector $(\sigma^2 \lij)\lij$ is denoted by $(\vsig)^2$ or $\vsig^2$. We will use $\vsig$ to denote $(\sigma\lij)\lij$

In each time slot, each input port can be matched to only one output port and similarly, each output port can be mapped to only one input port.
These constraints can be captured in a graph. Let $G$ denote a complete $n \times n $ bipartite graph with $n^2$ edges between the set of input ports and the set of output ports.
The schedule in each time slot is a matching on this graph $G$.
We let $s\lij=1$ if the link between input port $i$ and output port $j$ is matched or scheduled and $s\lij=0$ otherwise and we denote $\vs = (s\lij)\lij$. Then, the set of feasible schedules,
$\cS \subset \bR^{n^2}$ is defined as follows.
\ba
\cS = \left\{\vs \in \{0,1\}^{n^2} : \sum_{i=1}^{n}s\lij \leq 1, \sum_{j=1}^{n}s\lij \leq 1  \> \forall\> i,j\in\{1,2,\ldots,n\} \right\}.
\ea
Let $\cS^*$ denote the set of maximal feasible schedules. Then, it is easy to see that
\ba
\cS^* = \left\{\vs \in \{0,1\}^{n^2} : \sum_{i=1}^{n}s\lij = 1 , \sum_{j=1}^{n}s\lij = 1   \> \forall\> i,j\in\{1,2,\ldots,n\}\right\}.
\ea
Each element in this set corresponds to a perfect matching on the graph $G$. Each of these maximal feasible schedules is also a permutation $\pi$ on the set ${1,2,\ldots, n}$ with $\pi(i)=j$ if $s\lij=1$.

A scheduling policy or algorithm picks a schedule $\vs(t)$ in every time slot based on the current queue length vector, $\vq(t)$.
In each time slot, the order of events is as follows. Queue lengths at the beginning of time slot $t$ are $\vqt$. A schedule $\vs(t)$ is then picked for that time slot based on the queue lengths. Then, arrivals for that time $\va(t)$ happen. Finally the packets are served and there is unused service if there are no packets in a scheduled queue. The queue lengths are then updated to give the queue lengths for the next time slot. The queue lengths therefore evolve as follows.
\ba
q\lij(t+1)& = \left[\qijt+\aij(t)-\sij(t)\right]^+ \\
& = \qijt+\aij(t)-\sij(t)+\uij(t)\\
\vq(t+1) & = \vq(t)+\va(t)-\vs(t)+\vu(t)
\ea
where $[x]^+ = \max(0,x)$ is the projection onto positive real axis, $u\lij(t)$ is the unused service on link $(i,j)$. Unused service is $1$ only when link $(i,j)$ is scheduled, but has zero queue length; and it is $0$ in all other cases.
Thus, we have that when $u\lij(t)=1$, we have $\qij(t)=0$, $a\lij(t)=0$, $s\lij(t)=1$ and $\qij(t+1)=0$. 
Therefore, we have $u\lij(t)\qijt=0$, $u\lij(t)a\lij(t)=0$ and $u\lij(t)\qij(t+1)=0$.
Also note that since $u\lij(t)\leq s\lij(t)$, we have that $\sum_{i=1}^{n}u\lij \in \{0,1\}$ and $\sum_{j=1}^{n}u\lij\in \{0,1\}$ for all $i,j$.

The queue lengths process $\vq(t)$ is a Markov chain. The switch is said to be stable under a scheduling policy if the sum of all the queue lengths is finite, i.e.,
\ba
\limsup_{C\rightarrow \infty} \limsup_{t\rightarrow \infty} \bP\left(\sum\lij \qijt \geq C\right) =0.
\ea
If the queue lengths process $\vq(t)$ is positive recurrent under a scheduling policy, then we have stability.
The capacity region of the switch is the set of arrival rates \vlam for which the switch is stable under some scheduling policy. A policy that stabilizes the switch under any arrival rate in the capacity region is said to be throughput optimal. The MaxWeight Algorithm is a popular scheduling algorithm for the switches. In every time slot $t$, each link $(i,j)$ is given a weight equal to its queue length $q\lij(t)$ and the schedule with the maximum weight among the feasible schedules $\mathcal{S}$ is chosen at that time slot. This algorithm is presented in Algorithm \ref{alg:maxwt}. It is possible to show that the Markov chain $\vq(t)$ is irreducible and aperiodic under the MaxWeight algorithm for an appropriately defined state space \cite[Exercise 4.2]{srikantleibook}.
It is well known 
\cite{TasEph_92,mckeown96walrand} that the capacity region \cC of the switch is convex hull of all feasible schedules,
\ba
\cC = \text{Conv}(\mathcal{S}) = & \left\{\vlam \in \bR^{n^2}_+ : \sum_{i=1}^{n} \lambda\lij \leq 1  , \sum_{j=1}^{n} \lambda \lij \leq 1   \> \forall\> i,j\in\{1,2,\ldots,n\} \right\}\\
= & \left\{\vlam \in \bR^{n^2}_+ : \La\vlam,\vei \Ra  \leq 1, \La\vlam,\vetj \Ra  \leq 1 \> \forall\> i,j\in\{1,2,\ldots,n\} \right\}.
\ea
For any arrival rate vector $\vlam$, $\rho \triangleq \max\lij\{\sum_i\lambda\lij,\sum_j\lambda\lij\}$ is called the load.
It is also known that the queue lengths process is positive recurrent under the MaxWeight algorithm whenever the arrival rate is in the capacity region \cC (equivalently, load $\rho<1$) and therefore is throughput optimal.

\begin{algorithm}
\caption{\label{alg:maxwt}MaxWeight Scheduling Algorithm for an input-queued switch}

Consider the complete bipartite graph between the input ports and output ports. Let the queue length $q\lij(t)$ be the weight of the edge between input port $i$ and output port $j$. A maximum weight matching in this graph is chosen as the schedule in every time slot, i.e.,
\ban
\vs(t) = \argmax_{\vs \in \mathcal{S}} \sum\lij q\lij(t) s\lij = \argmax_{\vs \in \mathcal{S}} \left \langle \vq(t),\vs  \right\rangle \label{eq:MaxWt}
\ean
Ties are broken uniformly at random.
\end{algorithm}

Note that there is always a maximum weight schedule that is maximal. If the MaxWeight schedule chosen at time $t$, \vs is not maximal, there exists a maximal schedule $\vs^*\in \cS^*$ such that $\vs \leq \vs^*$ . For any link $(i,j)$ such that $s\lij=0$ and $s^*\lij=1$, $\qijt=0$. If not, \vs would not have been a maximum weight schedule. Therefore, we can pretend that the actual schedule chosen is $\vs^*$ and the links $(i,j)$ that are in \vs and $\vs^*$ have an unused service of $1$. Note that this does not change the scheduling algorithm, but it is just a convenience in the proof. Therefore, without loss of generality, we assume that the schedule chosen in each time slot is a maximal schedule, i.e.,
\ban
\vs(t)\in \cS^*   \text{ for all time }t \label{eq:maximal_sched}.
\ean Hence the MaxWeight schedule picks one of the $n!$ possible permutations from the set $\cS^*$ in each time slot.

For any arrival rate in the capacity region $\cC$, due to positive recurrence of $\vq(t)$, we have that a steady state distribution exists under MaxWeight policy. Let \bvq denote the steady state random vector.
In this paper, we focus on the average queue length under the steady state distribution, i.e., $\bE[\sum_{i,j}\bqij]$. We consider a set of systems indexed by $\epsilon$ with arrival rate $\vlam^{\epsilon}=(1-\epsilon)\vnu$, where \vnu is an arrival rate on the boundary of the capacity region \cC such that all the input and output ports are saturated and $\nu\lij>0$ for all $i,j$. The load of each system is then $(1-\epsilon)$.
We will study the switch when $\epsilon \downarrow 0$. This is called the heavy traffic limit.
We first show a 
universal lower bound on the average queue length in heavy traffic limit, i.e., on $\lim_{\epsilon \rightarrow 0} \bE[\sum_{i,j}\bqij]$. 
We then show that under MaxWeight policy, the limiting average queue length is 
within a factor of less than $2$ of the universal lower bound
and thus MaxWeight has optimal average queue length scaling. We will show these bounds using Lyapunov drift conditions.
We will use several different quadratic Lyapunov functions through out the paper.

\subsection{Geometry of the Capacity Region}

The capacity region $\cC $ is a coordinate convex polytope in $\bR^{n^2}$.
Here, we review some basic definitions.
For any set $P \in \bR^m$, its dimension is defined by 
\ba
\text{dim}(P) \triangleq \min \{\text{dim}(A) | P\subseteq A, A \text{ is an affine space } \}
\ea 
So the capacity region \cC has dimension $n^2$.
A hyperplane $H$  is said to be a supporting hyperplane of a polytope $P$ if $P\cap H \neq \emptyset$, $P\cap H_+ \neq \emptyset$ and $P\cap H_- = \emptyset$ where $H_+$ and $H_-$ are the open half-spaces determined by the hyperplane $H$. For any supporting hyperplane $H$ of polytope $P$, $P\cap H$ is called a face  \cite{Readdy_polytopes}. 
A face of a polytope is also a polytope with lower dimension. A face $F$ of polytope $P$ with dimension $\text{dim}(F)=\text{dim}(P)-1$ is called a \emph{facet}. Heavy traffic optimality of MaxWeight algorithm for generalized switches is shown in \cite{stolyar2004maxweight, erysri-heavytraffic} when a single input or output port is saturated or in other words when approaching an arrival rate vector on a facet of the capacity region. However, in this paper, we are interested in the case when all the ports are saturated. The arrival rate vector \vnu in this case does not lie on a facet and so, that result is not applicable here.

When \vnu is the arrival rate vector on the boundary of the capacity region such that all the input ports and all the output ports are saturated, it lies on the face \cF,
\ba
\cF = & \left\{\vlam \in \bR^{n^2}_+ : \sum_{i=1}^{n} \lambda\lij = 1 , \sum_{j=1}^{n} \lambda \lij = 1\> \forall\> i,j\in\{1,2,\ldots,n\}\right\}\\
= & \left\{\vlam \in \bR^{n^2}_+ : \La\vlam,\vei \Ra  = 1, \La\vlam,\vetj \Ra  =1 \> \forall\> i,j\in\{1,2,\ldots,n\} \right\}.
\ea
It is easy to see that \cF as defined here is indeed a face by observing that the hyperplane $\La \vlam, \vone\Ra = n$ is a supporting hyperplane of the capacity region \cC and it contains any rate vector \vnu where all the ports are saturated.
The face \cF has dimension $(n-1)^2= n^2-(2n-1)$, and lies in the affine space formed by the intersection of the $2n$ constraints, $\{\sum_{i=1}^{n} \lambda\lij = 1 \text{ for all } j\}$, and  $\{\sum_{j=1}^{n} \lambda \lij = 1 \text{ for all } i \} $. Of these $2n$ constraints, one is linearly dependent of the others and we have $2n-1$ linearly independent constraints.
The face \cF is actually the convex combination of the maximal feasible schedules $\cS^*$, i.e., $\cF = \text{Conv}(\cS^*)$.
These results follow from 
the fact that the face \cF 
is the Birkhoff polytope $B_n$ that contains all the $n\times n$ doubly stochastic matrices. It is known \cite[page 20]{ziegler1995lectures} that $B_n$ lies in the $(n-1)^2$ dimensional affine space of the constraints and is a convex hull of the permutation matrices.


A facet of a polytope has a unique supporting hyperplane defining the facet.
It was shown in \cite{erysri-heavytraffic} that when the arrival rate vector approaches a rate vector in the relative interior of a facet, in the limit, the queue length vector concentrates along the direction of the normal vector of the unique supporting hyperplane.
However, a lower dimensional face can be defined by one of several hyperplanes, and so there is no unique normal vector. A lower dimensional face is always an intersection of two or more facets. We are interested in the case when the arrival rate vector approaches the vector $\vnu$ that lies on the face \cF. The face \cF is the intersection of the $2n$ facets,
$\{\La \vei,\vlam \Ra=1\} \cap\cC$  for all $i$, and $\{\La \vetj,\vlam \Ra=1\} \cap \cC$  for all $j$.
We will show in section \ref{sec:SSC} that in the heavy traffic limit, the
queue length vector concentrates within the cone spanned by the $2n$ normal vectors, $\{ \vei \text { for all } i \} \cup \{ \vetj \text{ for all } j\}$. We will call this cone \cK.
Here, we will present some definitions and other results related to this cone.
More formally, the cone \cK can be defined as follows.
\ba
\cK = \left\{ \vx \in \bR^{n^2}: \vx=\sum_i w_i \vei +\sum_j \widetilde{w}_j \vetj \text{ where } w_i \in\bR_+ \text{ and } \widetilde{w}_j \in \bR_+ \text{ for all } i,j  \right\}
\ea

Note that this means that for any $\vx \in \cK$ there are $w_i \in\bR_+$ and $\widetilde{w}_j \in \bR_+$ for all $i,j \in \{1,2,\ldots,n\}$ such that $x\lij = w_i+\widetilde{w}_j $.
However, such a representation need not be unique. For example, suppose that $w_i\geq 1$ for all $i$ , then setting $w'_i=w_i-1$ for each $i$ and $\widetilde{w}'_j = \widetilde{w}_j+1$ for each $j$, we again have that $w'_i \in\bR_+$, $\widetilde{w}'_j \in \bR_+$ and $x\lij = w'_i+\widetilde{w}'_j $  for all $i,j$.

The cone \cK lies in the $2n-1$ dimensional subspace spanned by the $2n-1$ independent vectors among the $2n$ vectors, $\{ \vei \text { for all } i \} \cup \{ \vetj \text{ for all } j\}$.
Call this space $\mathcal{V}_{\cK}$.
For any two vectors $\vx,\vy \in \cF$, $\vx-\vy$ is orthogonal to the subspace $\mathcal{V}_{\cK}$, i.e.,
\ban
\vx-\vy \perp \mathcal{V}_{\cK} \label{eq:F_perp_K}.
\ean
This is easy to see since $\La\vx,\vei\Ra= \La\vy,\vei\Ra = 1$ for all $i$ and $\La\vx,\vetj\Ra= \La\vy,\vetj\Ra = 1$ for all $j$. If $\mathcal{V}_{\cF}$ denotes the subspace obtained by translating the affine space spanned by \cF, it follows that  the spaces $\mathcal{V}_{\cK}$ and $\mathcal{V}_{\cF}$ are orthogonal because translation means subtraction by a vector. Moreover, they span the entire space $\bR^{n^2}$ since their dimensions sum to $n^2$.
We now present a lemma about the structure of any vector in the cone $\cK$.
\begin{lemma}\label{lem:cone_q_avg}
For any vector $\vx \in \cK$, we have
\ba
x\lij = \frac{1}{n}\sum_{j'=1}^n x_{ij'} + \frac{1}{n} \sum_{i'=1}^n x_{i'j} -\frac{1}{n^2} \sum_{i'=1}^n\sum_{j'=1}^n x_{i'j'}.
\ea
\end{lemma}
\begin{remark}
In other words, any element in the matrix $\vx$, is equal to the average of all the elements in its row plus the average of all the elements in its column minus the average of all the elements in the whole matrix. Suppose the queue lengths $\vq\in\cK$, then any queue length from an input port to an output port is equal to the average queue lengths at that input port plus the average queue lengths at that output port minus the average queue length of all the queues in the switch.
\end{remark}
\begin{proof}
Since $x\lij$ is of the form $w_i+\widetilde{w}_j$ for any $\vx \in \cK$, we have
\ba
 \frac{1}{n}\sum_{j'=1}^n x_{ij'} + \frac{1}{n} \sum_{i'=1}^n x_{i'j} -\frac{1}{n^2} \sum_{i'=1}^n\sum_{j'=1}^n x_{i'j'} = &\frac{1}{n} \sum_{j'=1}^n (w_i+\widetilde{w}_{j'} ) + \frac{1}{n} \sum_{i'=1}^n (w_{i'}+\widetilde{w}_j ) -\frac{1}{n^2} \sum_{i'=1}^n\sum_{j'=1}^n (w_{i'}+\widetilde{w}_{j'} )\nonumber\\
=& w_i+\widetilde{w}_{j} \nonumber\\
=& x\lij.
\ea
This proves the lemma.
\end{proof}

\subsubsection{Projection onto the cone \cK}

The cone \cK is closed and convex. For any $\vx \in \bR^{n^2}$, the closest point in the cone \cK to $\vx$ is called the projection of $\vx$ on to the cone \cK and we will denote it by
$\vx\para$. More formally,
\ba
\vx\para = \argmin_{\vy \in \cK}\|\vx-\vy \|
\ea
For a closed convex cone \cK,  the projection $\vx\para$ is well defined and is unique 
\cite[Appendix E.9.2]{mebooconvex}. 
We will use $\vx\per$ to denote $\vx-\vx\para$.
We will use $x\paraij $ to denote the $(i,j)^{\text{th}}$ component of $\vx\para$. Similarly, $x\perij $.

Note that unlike projection on to a subspace, projection on to a cone is not linear, i.e., $(\vx+\vy)\para\neq \vx\para+\vy\para$. A simple counter example is the following. In $\bR^2$, let $\vx=(2,2)$ and $\vy=(-1,-1)$. Consider the positive quadrant as the cone of interest. Then, $\vx\para=(2,2)$, $\vy\para = (0,0)$ and $(\vx+\vy)\para = (1,1)$.

Since for any $\vx\in \bR^{n^2}$, $\vx\para \in \cK$, from the definition of the cone \cK, we have that every component of $\vx\para$ is non negative, i.e., $x\paraij \geq 0$. However, $\vx\per$ could have negative components.

The polar cone $\cKo$ of cone \cK is defined as
\ba
\cKo= \left\{\vx \in \bR^{n^2}: \La\vx, \vy\Ra\leq 0 \text{ for all } \vy \in \cK\right\}.
\ea
The polar cone $\cKo$  is negative of the dual cone $\cK^*$ of the cone \cK. For any $\vx \in \bR^{n^2}$, $\vx\per \in \cKo$ and $\La\vx\para,\vx\per\Ra=0$ \cite[Appendix E.9.2]{mebooconvex}. Therefore, pythagoras theorem is applicable, i.e.,
\ban
\|\vx\|^2=\|\vx\para\|^2+\|\vx\per\|^2 \label{eq:pythagoras}
\ean and so, $\|\vx\para\| \leq \|\vx\|$ and $\|\vx\per\| \leq \|\vx\|$.

Projection onto any closed convex set in $\bR^{n^2}$ (and so onto a closed convex cone) is nonexpansive \cite[Appendix E.9.3]{mebooconvex}. Therefore, we have $\|\vx\para-\vy\para| \leq \|\vx-\vy\|$. Since $\vx\per$ is a projection onto $\cKo$, we also have
\ban
\|\vx\per-\vy\per| \leq \|\vx-\vy\| \label{eq:nonexpansive_projection}.
\ean

\subsection{Moment bounds from Lyapunov drift conditions}
In this paper, we will use the Lyapunov drift based approach presented in \cite{erysri-heavytraffic} to obtain bounds of average queue length under MaxWeight. A key ingredient in this approach is to obtain moment bounds from drift conditions. A lemma from \cite{hajek_drift} was used in \cite{erysri-heavytraffic} to obtain these bounds and we first state it here as it was stated in \cite{erysri-heavytraffic}.

\begin{lemma} \label{lem:Hajek}
For an irreducible and aperiodic Markov chain $\{X(t)\}_{t\geq 0}$ over a countable state space $\mathcal{X},$ suppose $Z:\mathcal{X}\rightarrow \mathbb{R}_+$ is a nonnegative-valued Lyapunov function. We define the drift of $Z$ at $X$ as $$\Delta Z(X) \triangleq [Z(X(t+1))-Z(X(t))]\>\mathcal{I}(X(t)=X),$$ where $\mathcal{I}(.)$ is the indicator function. Thus, $\Delta Z(X)$ is a random variable that measures the amount of change in the value of $Z$ in one step, starting from state $X.$ This drift is assumed to satisfy the following conditions:
\begin{enumerate}[label=\textbf{C\arabic*},ref=C.\arabic*] 
\item\label{cond:C1} There exists an $\eta>0,$ and a $\kappa<\infty$ such that for any $t=1,2,\ldots$ and for all $X\in \mathcal{X}$ with $ Z(X)\geq \kappa,$
 \begin{eqnarray*}
    \mathbb{E}[\Delta Z(X) | X(t)=X]
    \leq - \eta .
 \end{eqnarray*}
\item\label{cond:C2} There exists a $D < \infty$ such that for all $X \in \mathcal{X},$
  \begin{eqnarray*}
  \mathbb{P}\left(|\Delta Z(X)| \leq D\right) = 1.
  \end{eqnarray*}
\end{enumerate}
Then, there exists a $\theta^\star > 0$ and a $C^\star < \infty$ such that
$$\limsup_{t\rightarrow \infty} \mathbb{E}\left[e^{\theta^\star Z(X(t))}\right] \leq C^\star.$$
If we further assume that the Markov chain $\{X(t)\}_t$ is positive recurrent, then $Z(X(t))$ converges in distribution to a random variable $\overline{Z}$ for which
$$ \mathbb{E}\left[e^{\theta^\star \overline{Z}}\right] \leq C^\star,$$
which directly implies that all moments of $\overline{Z}$ exist and are finite.
\end{lemma}

This lemma (and its original form in \cite{hajek_drift}) is quiet general and versatile. However, we use a different result in this paper to obtain moment bounds that are tighter than the bounds that can be obtained using Lemma \ref{lem:Hajek} (or its original form in \cite{hajek_drift}).
The following lemma essentially follows from \cite[Theorem 1]{bertsimas_momentbound} except for some minor differences. The proof is presented in Appendix \ref{sec:Bertsimas_proof} and makes use of  Lemma \ref{lem:Hajek}.

\begin{lemma} \label{lem:Bertsimas} Consider an irreducible and aperiodic Markov chain Markov Chain $\{X(t)\}_{t\geq 0}$ over a countable state space $\mathcal{X},$
 suppose $Z:\mathcal{X}\rightarrow \mathbb{R}_+$ is a nonnegative-valued Lyapunov function.
The drift  $\Delta Z(X)$ of $Z$ at $X$  as defined in Lemma \ref{lem:Hajek} is assumed to satisfy the conditions \ref{cond:C1} and \ref{cond:C2}
%
Further assume that the Markov chain $\{X(t)\}_t$ converges in distribution to a random variable $\overline{X}$. 
Then, for any $m=0,1,2,\ldots$,
\ba
\bP\left(Z\left(\overline{X}\right)>\kappa+2Dm\right)\leq \left(\frac{D}{D+\eta} \right)^{m+1}.\ea
As a result, for any $r=1,2,\ldots$,
\ba
\bE[Z\left(\overline{X}\right)^r]\leq &
\left(2\kappa\right)^r+
 \left(4D\right)^r \left(\frac{D+\eta}{\eta} \right)^r r!.
\ea
\end{lemma}

\section{Universal Lower Bound}\label{sec:ULB}
In this section, we will prove the following lower bound on the average queue lengths, which is valid under any scheduling policy.

\begin{proposition}\label{prop:Universal_LB}
Consider a set of switch systems with the arrival processes
$\va\se(t)$ described in Section \ref{sub:model}, parameterized by $0<\epsilon<1$, such that the mean arrival rate vector is $\vlam^{\epsilon}=(1-\epsilon)\vnu$ for some $\vnu \in \cF$ and variance is $\left(\vsig^{(\epsilon)}\right)^2$. The load is then $\rho=(1-\epsilon)$.
Fix a scheduling policy under which the switch system is stable for any $0<\epsilon<1$.
Let $\vq^{(\epsilon)}(t)$ denote the queue lengths process under this policy for each system. Suppose that this process converges in distribution to a steady state random vector $\bvq^{(\epsilon)}$.
Then, for each of these systems, the average queue length is lower bounded by
\ba
\bE \left [\sum\lij \bqij \se \right ] \geq \frac{\left\|\vsig\se\right\|^2}{2\epsilon} - \frac{n(1-\epsilon)}{2}
\ea
Therefore, in the heavy-traffic limit as $\epsilon \downarrow 0$, if $\left(\vsig^{(\epsilon)}\right)^2 \rightarrow \vsig^2$, we have
\ba
\liminf_{\epsilon \downarrow 0} \epsilon \bE \left [\sum\lij \bqij \se \right ] \geq \frac{\left\|\vsig\right\|^2}{2} .
\ea
\end{proposition}
\begin{proof}
We will obtain a lower bound on sum of all the queue lengths by lower bounding the queue lengths at each input port, i.e., we will first bound $\bE[\sum_j \bqij\se]$ for a fixed input port $i$. We do this by considering a single queue that is coupled to the process $\sum_j \qij\se(t)$. Consider a single server queue $\phi_i\se(t)$ in discrete time. Packets 
arrive into this queue to be served. Each packet needs exactly one time slot of service. The arrival process to this queue is $\alpha_i\se(t)=\sum_ja\lij\se(t)$. Mean arrival to this queue is $\bE[\alpha_i\se(t)]=\sum_j\lambda\lij\se(t)=(1-\epsilon)\sum_j\nu\lij\se(t)=(1-\epsilon)$ since $\vnu \in \cF$.. As long as the queue is non empty, one packet is served in every time slot. Thus, this queue evolves as
\ba
\phi_i\se(t+1) &= \left[ \phi_i\se(t)+ \alpha_i\se(t) -1 \right ]^+ \\
&=  \phi_i\se(t)+ \alpha_i\se(t) -1 + \upsilon\se(t)
\ea
where $\upsilon\se(t)$ is the unused service and so $\upsilon\se(t)\phi_i\se(t+1)=0$.
Clearly, $\phi_i\se(t)$ is positive recurrent and let $\overline{\phi}_i\se$ denote the steady state random variable to which it converges in distribution.
\begin{claim}
In steady state
\ba
\bE[\sum_j \bqij\se] \geq \bE[\overline{\phi}_i\se]
\ea
\end{claim}
\begin{proof}
Suppose that at time zero, the queue $\phi_i\se$ starts with $\phi_i\se(0)= \sum_j \qij\se(0)$.
Then, for any time $t$, the queue $\phi_i\se(t)$ is stochastically no greater than $\sum_j \qij\se(t)$.
This can easily be seen using induction. For $t=0$, we have $\sum_j \qij(0) \geq \phi_i(0)$. Suppose that $\sum_j \qij\se(t) \geq \phi_i\se(t)$. Then, at time $(t+1)$,
\ba
\sum_j q\lij\se(t+1)& = \sum_j \left[\qij\se(t)+\aij\se(t)-\sij\se(t)\right]^+ \\
& \geq  \left[\sum_j ( \qij\se(t)+\aij\se(t)-\sij\se(t)) \right]^+  \\
& \geq  \left[\phi_i\se(t) + \alpha_i\se(t) -1 ) \right]^+  \\
& = \phi_i\se(t+1)
\ea
where the last inequality follows from the inductive hypothesis, definition of $\alpha\se(t)$, the constraint $\sij\se(t) \leq 1$ and the fact that if $x\geq y$, we have that $[x]^+ \geq [y]^+$.
Thus, we have that in steady state, $\bE[\sum_j \bqij\se] \geq \bE[\overline{\phi}_i\se]$. Since steady state distribution of $\sum_j \bqij\se$ and $\overline{\phi}_i\se$ does not depend on the initial state at time zero, we have the lower bound $\bE[\sum_j \bqij\se] \geq \bE[\overline{\phi}_i\se]$ independent of the initials states $\phi_i\se(0)$ and $\sum_j \qij\se(0)$.
\end{proof}


We will now bound $\bE[\overline{\phi}_i\se]$. This result is obtained in \cite{srikantleibook}. We present it here  for completeness. Consider the drift of $\bE[(\phi_i\se(t))^2]$.
\ba
\bE[(\phi_i\se(t+1))^2- (\phi_i\se(t))^2 ]
& = \bE[(\phi_i\se(t)+ \alpha_i\se(t) -1 + \upsilon\se(t))^2- (\phi_i(t)\se)^2 ]\\
& \stackrel{(a)}{=} \bE[(\phi_i\se(t)+ \alpha_i\se(t) -1 )^2 - (\upsilon\se(t))^2- (\phi_i(t)\se)^2 ]\\
& = \bE[(\alpha_i\se(t) -1 )^2 + 2(\phi_i\se(t))(\alpha_i\se(t) -1 )- (\upsilon\se(t))^2]\\
& \stackrel{(b)}{=} \bE[(\alpha_i\se(t) -(1-\epsilon)-\epsilon )^2] -2\epsilon \bE[\phi_i\se(t)]- \bE[\upsilon\se(t)]\\
& \stackrel{(c)}{=} \text{Var}\left(\alpha_i\se(t) \right)+\epsilon^2 -2\epsilon \bE[\phi_i\se(t)]- \bE[\upsilon\se(t)]\\
& \stackrel{(d)}{=} \sum_j \left(\sigma\lij\se\right)^2+\epsilon^2 -2\epsilon \bE[\phi_i\se(t)]- \bE[\upsilon\se(t)]
\ea
where $(a)$ follows from noting that $(\upsilon\se(t))(\phi_i\se(t)+ \alpha_i\se(t) -1 + \upsilon\se(t))=0$; (b) follows from independence of $\phi_i\se(t)$ and the arrivals $\alpha_i\se(t)$ and since $\upsilon\se(t)\in\{0,1\}$; (c) follows from the fact that $\bE[\phi_i\se(t)]=(1-\epsilon)$; (d) follows from the definition of $\alpha_i\se(t)$ and independence of the arrival process $a\lij(t)$ across ports.
It can easily be shown that  $\bE[(\overline{\phi}_i\se)^2]$ is finite \cite[Section 10.1]{srikantleibook}. Therefore, the steady state drift of $\bE[(\phi_i\se(t))^2]$ is zero, i.e., in steady-state, we get
\ban
2\epsilon \bE[\overline{\phi}_i\se]= \sum_j \left(\sigma\lij\se\right)^2+\epsilon^2 - \bE[\overline{\upsilon}\se] \label{eq:LB_single_server}
\ean
Consider the drift of $\bE[\phi_i\se(t)]$.
\ba
\bE[\phi_i\se(t+1)- \phi_i\se(t) ]
& = \bE[\alpha_i\se(t) -1 + \upsilon\se(t) ]\\
& = -\epsilon+\bE[ \upsilon\se(t) ]
\ea
Since $\phi_i\se(t) \in \bZ_+$, we have $\phi_i\se(t) \leq (\phi_i\se(t))^2$, and so we get finiteness of $\bE[\overline{\phi}_i\se]$ from that of $\bE[(\overline{\phi}_i\se)^2]$. Therefore, the drift of $\bE[\phi_i\se(t)]$ is zero in steady state. Thus, we get that in steady state, $\bE[\overline{\upsilon}\se]=\epsilon$. Substituting this in (\ref{eq:LB_single_server}), and using the claim, we get
\ban
\bE[\sum_j \bqij\se] \geq \bE[\overline{\phi}_i\se] = \frac{1}{2\epsilon}\sum_j \left(\sigma\lij\se\right)^2 - \frac{1-\epsilon}{2} \label{eq:LB_single_server2}
\ean
Since this lower bound is true for any input port $i$, summing over all the input ports, we have the proposition. Note that we could have obtained the same bound by similarly lower bounding the sum of lengths of all the queues destined to port $j$, i.e., $\sum_i \qij\se(t)$.
\end{proof}
We do not know if this lower bound is tight, i.e., if there is a scheduling policy that attains this lower bound. However, in section \ref{sec:MaxWt_UB}, we show that under MaxWeight scheduling algorithm, the average queue lengths are within a factor of less than $2$ away from this universal lower bound, thus showing that MaxWeight has optimal scaling.
Closing this gap is an open question.

\section{State Space Collapse under MaxWeight policy}\label{sec:SSC}

In this section, we will show that under the MaxWeight scheduling algorithm, in the heavy traffic limit, the steady state queue length vector concentrates within the cone \cK in the following sense.
As the parameter $\epsilon$ approaches zero, the mean arrival rate approaches the boundary of the capacity region and we know from the lower bound  that the average queue lengths go to infinity $\Omega(1/\epsilon)$.
We will show that under the MaxWeight algorithm, the  $\bvq\se\per$
component of the queue length vector is upper bounded independent of $\epsilon$. Thus the $\bvq\se\per$ component is negligible compared to the $\bvq\se\para$ component of $\bvq\se$.
This is called state space collapse. We say that the state space collapses to the cone $\cK$.
It was shown in \cite{rui_567classproj} that the state space collapses to the subspace containing the cone $\cK$. A similar result was also shown in \cite{singh2015maxweight} for a different problem.
Here, we show the stronger result that the state space collapses to the cone, which is essential to obtain the upper bounds in Section \ref{sec:MaxWt_UB}.

We define the following Lyapunov functions and their corresponding drifts.
\begin{gather*}
\Vq \triangleq \|\vq\|^2= \sum\lij \qij^2 \quad
\Wperq \triangleq \|\vq\per\| \qquad
\Vperq \triangleq \|\vq\per\|^2= \sum\lij q\perij^2 \qquad
\Vparaq \triangleq \|\vq\para\|^2= \sum\lij q\paraij^2 \\
\begin{aligned}
\Delta \Vq &\triangleq [\V(\vq(t+1))-\V(\vq(t))]\>\mathcal{I}(\vq(t)=\vq) \nonumber\\
\Delta \Wperq &\triangleq [\Wper(\vq(t+1))-\Wper(\vq(t))]\>\mathcal{I}(\vq(t)=\vq)  \\
\Delta \Vperq &\triangleq [\Vper(\vq(t+1))-\Vper(\vq(t))]\>\mathcal{I}(\vq(t)=\vq) \\
\Delta \Vparaq &\triangleq [\Vpara(\vq(t+1))-\Vpara(\vq(t))]\>\mathcal{I}(\vq(t)=\vq)
\end{aligned}
\end{gather*}

We will use Lemma \ref{lem:Bertsimas} using the Lyapunov function $\Wperq(.)$ to bound the $\bvq\se\per$ component in steady state. We need the following lemma, which follows from concavity of square root function and the pythagorean theorem (\ref{eq:pythagoras}). The proof of this lemma is similar to the proof of Lemma 7 in \cite{erysri-heavytraffic} and so we skip it here.

\begin{lemma}\label{lem:Wperp_V}
Drift of $\Wper(.)$ can be bounded in terms of drift of $\V(.)$ and $\Vpara(.)$ as follows.
\ba
\Delta \Wperq \leq \frac{1}{2\|\vq\per\|}\left(\Delta \Vq - \Delta \Vparaq \right) \qquad \forall \vq\in\bR^{n^2}
\ea
\end{lemma}

We will now formally state the state space collapse result.
\begin{proposition} \label{prop:SSC}
Consider a set of switch systems under MaxWeight scheduling algorithm, with the arrival processes
$\va\se(t)$  described in Section \ref{sub:model}, parameterized by $0<\epsilon<1$, such that the mean arrival rate vector is $\vlam^{\epsilon}=(1-\epsilon)\vnu$ for some $\vnu \in \cF$ such that $\numin\triangleq \min\lij \nu\lij >0$.
 The load is then $\rho=(1-\epsilon)$.
Let the variance of the arrival process  be $\left(\vsig^{(\epsilon)}\right)^2$. 
Let $\vq^{(\epsilon)}(t)$ denote the queue lengths process of each system, which is positive recurrent. Therefore, the process $\vq^{(\epsilon)}(t)$ converges to a steady state random vector in distribution, which we denote by $\bvq^{(\epsilon)}$.
Then, 
for each system with $0< \epsilon \leq \numin/2\|\vnu\|$, the steady state queue lengths vector satisfies
\ba
\bE \left [\| \bvq\per \se \|^r\right ] \leq (M_r\se)^r \>\> \forall r \in\{1,2,\ldots\},
\ea
where
\ba
M_r\se= 2^{\frac{1}{r}} \max \left(\frac{8(\|\vlam\se\|^2+\|\vsig\se\|^2+n)}{\numin},
(\sqrt{r}e)^{1/r}
16\frac{r}{e}\frac{na_{\max}}{\numin} \left(na_{\max}+1 \right) \right).
\ea

%
\end{proposition}
\begin{remark}
Note that for any $r$, the expressions $M_r$ are upper bounded by a constant not dependant on $\epsilon$ whenever there exists a $\widetilde{\sigma}$ which does not depend on $\epsilon$ such that $\left(\vsig^{(\epsilon)}\right)^2 \leq \widetilde{\sigma}$ for all $\epsilon$. This is why we call this state space collapse. Our notion of state-space collapse considers the system in steady-state, and is hence mathematically different from the state-space collapse result in \cite{ShaWis_11}, although the results are similar in spirit.
\end{remark}
\begin{proof}
We will skip the superscript $\se$ in this proof for ease of notation. Thus, we will use \vqt, \vlam and \vsig to denote $\vq\se(t)$, $\vlam\se$ and $\vsig\se$ respectively.
We will verify both the conditions \ref{cond:C1} and \ref{cond:C2} to apply Lemma \ref{lem:Bertsimas} for the Markov chain \vqt and Lyapunov function $\Wper (\vq(\cdot))$. First we consider condition \ref{cond:C2}.
\ban
| \Delta \Wperq| =& |\|\vq\per(t+1)\|-\|\vq\per(t)\||\>\mathcal{I}(\vq(t)=\vq) \nonumber\\
\stackrel{(a)}{\leq} & \|\vq\per(t+1)-\vq\per(t)\|  \nonumber\\
\stackrel{(b)}{\leq} & \|\vq(t+1)-\vq(t)\| \nonumber\\
=&\sqrt{\sum\lij |\qij(t+1)-\qij(t)|^2} \nonumber\\
\stackrel{(c)}{\leq}&\sqrt{\sum\lij a_{\max}^2} \nonumber\\
\leq &n a_{\max}  \label{eq:delta_Wperp}
\ean
where (a) follows from triangle inequality, i.e., $|\|\vx\|-\|\vy\||\leq \|\vx-\vy\|$ and $\mathcal{I}(.)\leq 1$; (b) follows from nonexpansivity of projection operator (\ref{eq:nonexpansive_projection}); (c) is true because each queue lengths can increase by at most $a_{\max}\geq 1$ due to arrivals and can decrease by at most $1$ due to departures. Thus condition \ref{cond:C2} of Lemma 
is true with $D=n a_{\max}$.

We will now verify \ref{cond:C1}, using Lemma \ref{lem:Wperp_V} by bounding the drifts $\Delta \Vq$ and $\Delta \Vparaq$. 
\ban
\lefteqn{
\bEq{\Delta\Vq}}\nonumber\\
  =& \bEq{\|\vq(t+1)\|^2-\|\vq(t)\|^2} \nonumber\\
= & \bEq{\|\vqt+\va(t)-\vs(t)+\vu(t)\|^2-\|\vqt\|^2} \nonumber\\
= & \bEq{\|\vqt+\va(t)-\vs(t)\|^2+\|\vu(t)\|^2+2\La\vq(t+1)-\vu(t),\vu(t)\Ra-\|\vqt\|^2} \nonumber\\
\stackrel{(a)}{\leq} & \bEq{\|\va(t)-\vs(t)\|^2+2\La\vq(t),\va(t)-\vs(t)\Ra} \nonumber\\
\stackrel{(b)}{=}&  \bEq{\sum\lij ( a\lij^2(t)+s\lij(t)-2a\lij(t)s\lij(t))}+2\La\vq,\vlam-\bEq{\vs(t)}\Ra \nonumber\\
\stackrel{(c)}{=}&  \sum\lij( \lambda\lij^2+\sigma\lij^2)+n -2\bEq{\sum\lij \lambda\lij(t)s\lij(t)} +2\La\vq,\vlam-\bEq{\vs(t)}\Ra \label{eq:delta_V}\\
=&   \|\vlam\|^2+\|\vsig\|^2+n -2(1-\epsilon)\bEq{\sum\lij \nu\lij s\lij(t)}
+2\La\vq,(1-\epsilon)\vnu-\bEq{\vs(t)}\Ra \nonumber\\
\leq& \|\vlam\|^2+\|\vsig\|^2+n
 -2\epsilon \La\vq,\vnu\Ra
+2\La\vq,\vnu-\bEq{\vs(t)}\Ra \nonumber\\
=& \|\vlam\|^2+\|\vsig\|^2+n
 -2\epsilon \La\vq,\vnu\Ra
+2\min_{\bf{r}\in \cC} \La\vq,\vnu-\bf{r}\Ra\label{eq:SSCdelV_intermediate}
\ean
where (a) follows from the fact that $\La\vq(t+1),\vu(t)\Ra=0$ and dropping the $-\|\vu(t)\|^2$ term; (b) is true because $s\lij\in\{0,1\}$. Note that $\bE[a\lij^2(t)]=\bE[a\lij(t)]^2+\text{Var}(a\lij(t))$. Also note that the arrivals in each time slot are independent of the queue lengths and hence are also independent of the service process. These facts and (\ref{eq:maximal_sched}) give (c). 
Since we use MaxWeight scheduling algorithm, from (\ref{eq:MaxWt}), we have (\ref{eq:SSCdelV_intermediate}). In order to bound the last term in (\ref{eq:SSCdelV_intermediate}), we present the following claim.
\begin{claim}
For any $\vq \in \bR^{n^2}$ ,
\ba
\vnu+\frac{\numin}{\|\vq\per\|}\vq\per \in \cC.
\ea
\end{claim}
\begin{proof}
Since $\left|q\perij\right| \leq \|\vq\per\|$, $\nu\lij+\frac{\numin}{\|\vq\per\|}q\perij \geq \nu\lij-\numin\geq0$ and so ${\vnu+\frac{\numin}{\|\vq\per\|}\vq\per \in \bR^{n^2}_+}$. We know that $\vq\per \in \cKo$ and $\vei \in \cK$, and so $\La\vq\per,\vei \Ra \leq 0$. Thus, for any $i$, we have
\ba
\La\vnu+\frac{\numin}{\|\vq\per\|}\vq\per ,\vei \Ra =& \La\vnu,\vei \Ra+ \frac{\numin}{\|\vq\per\|}\La\vq\per ,\vei \Ra \\
\leq & \La\vnu,\vei \Ra \\
 = & 1
\ea
where the last equality is due to the fact that $\vnu \in \cF$.
Similarly, we can show that ${\La\vnu+\frac{\numin}{\|\vq\per\|}\vq\per ,\vetj \Ra \leq 1}$ for any $j$, proving the claim.
\end{proof}
Using the claim in (\ref{eq:SSCdelV_intermediate}), we get
\ban
\bEq{\Delta\Vq}
\leq&  \|\vlam\|^2+\|\vsig\|^2+n -2\epsilon \La\vq,\vnu\Ra
+2\La\vq,\vnu-\left(\vnu+\frac{\numin}{\|\vq\per\|}\vq\per\right)\Ra \nonumber\\
= & \|\vlam\|^2+\|\vsig\|^2+n -2\epsilon \La\vq,\vnu\Ra
-\frac{2\numin}{\|\vq\per\|} \La\vq\para+\vq\per,\vq\per\Ra \nonumber\\
= & \|\vlam\|^2+\|\vsig\|^2+n -2\epsilon \La\vq,\vnu\Ra
-2\numin\|\vq\per\| \label{eq:SSCdelV}
\ean
where the last equality follows from the fact that $\La\vq\para,\vq\per\Ra=0$. We will now bound the drift $\Delta \Vparaq$.
\ban
\lefteqn{\bEq{\Delta\Vparaq}} \nonumber\\
  =& \bEq{\|\vq\para(t+1)\|^2-\|\vq\para(t)\|^2} \nonumber\\
=& \bEq{\La\vq\para(t+1)+\|\vq\para(t),\vq\para(t+1)-\|\vq\para(t)\Ra} \nonumber\\
=& \bEq{\|\vq\para(t+1)-\vq\para(t)\|^2} +2\bEq{\La\vq\para(t),\vq\para(t+1)-\vq\para(t)\Ra} \nonumber\\
\geq& 2\bEq{\La\vq\para(t),\vq\para(t+1)-\vq\para(t)\Ra} \nonumber\\
=&2\bEq{\La\vq\para(t),\vq(t+1)-\vq(t)\Ra} - 2\bEq{\La\vq\para(t),\vq\per(t+1)-\vq\per(t)\Ra} \nonumber\\
\stackrel{(a)}{\geq} & 2\bEq{\La\vq\para(t),\va(t)-\vs(t)+\vu(t)\Ra} \nonumber\\
\stackrel{(b)}{\geq} & 2\La\vq\para,\vlam\Ra- 2\bEq{\La\vq\para,\vs(t)\Ra} \nonumber\\
= & -2\epsilon\La\vq\para,\vnu\Ra- 2\bEq{\La\vq\para,\vs(t)-\vnu\Ra} \nonumber\\
= & -2\epsilon\La\vq\para,\vnu\Ra \label{eq:SSCdelVpara}
\ean
Equation (a) is true because $\La\vq\para(t),\vq\per(t)\Ra=0$ and $\La\vq\para(t),\vq\per(t+1)\Ra \leq 0$ since $\vq\para(t) \in \cK$ and $\vq\per(t+1) \in \cKo$. All the components of $\vq\para$ and $\vu(t)$ are nonnegative. Using this fact with independence of the arrivals and the queue lengths gives Equation (b). The last equality follows from (\ref{eq:F_perp_K}) since $\vq\para \in \cK \in \mathcal{V}_{\cK}$ and $\vs(t),\vnu \in \cF$ from (\ref{eq:maximal_sched}). Now substituting (\ref{eq:SSCdelV}) and (\ref{eq:SSCdelVpara}) in Lemma \ref{lem:Wperp_V}, we get
\ba
\bEq{\Delta \Wperq} \leq & \frac{1}{2\|\vq\per\|}\left(\|\vlam\|^2+\|\vsig\|^2+n -2\epsilon \La\vq,\vnu\Ra
-2\numin\|\vq\per\| +2\epsilon\La\vq\para,\vnu\Ra \right) \\
= & \frac{\|\vlam\|^2+\|\vsig\|^2+n}{\|\vq\per\|}-\numin-\frac{\epsilon}{\|\vq\per\|}
\La\vq\per,\vnu\Ra\\
\stackrel{(a)}{\leq}  &  \frac{\|\vlam\|^2+\|\vsig\|^2+n}{\|\vq\per\|}-\numin+\epsilon \|\vnu\|\\
\leq  & \frac{\|\vlam\|^2+\|\vsig\|^2+n}{\|\vq\per\|}-\frac{\numin}{2} \text{ whenever } \epsilon \leq \frac{\numin}{2\|\vnu\|} \\
\leq & - \frac{\numin}{4} \text{ for all } \vq \text{ such that } \Wperq\geq \frac{4(\|\vlam\|^2+\|\vsig\|^2+n)}{\numin}
 \ea
where (a) is due to the Cauchy Schwartz inequality $\La\frac{-\vq\per}{\|\vq\per\|},\vnu\Ra \leq \frac{\|\vq\per\|}{\|\vq\per\|}\|\vnu\|$. 
 Thus condition \ref{cond:C1} is valid with $\kappa = \frac{4(\|\vlam\|^2+\|\vsig\|^2+n)}{\numin}$ and $\eta = \frac{\numin}{4}$. Then from Lemma \ref{lem:Bertsimas}, we get for $r=1,2,\ldots$,
 \ba
\bE \left [\| \bvq\per \se \|^r\right ]
\leq & \left(\frac{8(\|\vlam\|^2+\|\vsig\|^2+n)}{\numin}\right)^r+
r!
 \left(16\frac{na_{\max}}{\numin}\right)^r \left(na_{\max}+\frac{\numin}{4} \right)^r \\
\stackrel{(a)}{ \leq }& \left(\frac{8(\|\vlam\|^2+\|\vsig\|^2+n)}{\numin}\right)^r+
\sqrt{r}e
 \left(16\frac{r}{e}\frac{na_{\max}}{\numin} \left(na_{\max}+1 \right) \right)^r  \\
  \leq & 2 \max \left(\frac{8(\|\vlam\|^2+\|\vsig\|^2+n)}{\numin},
(\sqrt{r}e)^{1/r}
16\frac{r}{e}\frac{na_{\max}}{\numin} \left(na_{\max}+1 \right) \right)^r.
\ea
where (a) follows from Stirling's upper bound of the factorial function, $r!\leq e^{1-r} r^{r+\frac{1}{2}}$ and noting that $\numin\leq 1$ follows from the definition of $\numin$ and the capacity region $\cC$. 
The last inequality follows from $a^r+b^r\leq 2 \max(a,b)^r$, proving the proposition.
\end{proof}

Recall that there are $n!$ maximal schedules (permutations or perfect matchings). For each of them, MaxWeight assigns a weight which is the sum of corresponding queue lengths and then picks the one with the maximum weight. In this process, it is equalizing the weights of all the schedules by serving the matching with maximum weight and thereby decreasing it. The cone \cK has the property that if the queue lengths vector \vq is in the cone \cK, we have $w_i$ and $\widetilde{w}_j$ such that $\qij= w_i+\widetilde{w}_j$. This means that all the maximal schedules have the same weight $\sum_iw_i+\sum_j\widetilde{w}_j$ and the MaxWeight algorithm is agnostic between them. Thus, the state space collapse result states that in steady state, MaxWeight is (almost) successful in being able to equalize the weights of all maximal schedules in the heavy traffic limit.
This behavior is very similar to Join-the-shortest queue (JSQ) routing policy in a supermarket checkout system. In such a system, there are a few servers, each with a queue. When a customer arrives to be served, under JSQ policy, (s)he picks the server with the shortest queue. It was shown in \cite{erysri-heavytraffic} that in the heavy traffic limit, the state of this system collapses to a state where all the queues are equal, and thus, JSQ is agnostic between all the queues when such a state space collapse occurs. Here JSQ policy is trying to equalize all the queues by increasing the shortest one, and it is (almost) successful in doing that in steady state in heavy traffic limit.

A natural question in this context is if there is any interpretation to the variables $w_i$ and $\widetilde{w}_j$. These variables 
are the optimal dual variables for the maximum weight matching problem. The maximum weighted perfect matching problem in bipartite graphs (that MaxWeight solves in every time slot) can be written as the integer program (\ref{eq:IP}) and its linear program (LP) relaxation is the linear program (\ref{eq:LP_relaxation}).

\noindent\begin{minipage}{.5\linewidth}
\ban
\max & \sum\lij \qij\sij  \nonumber\\
\text{subject to: } & \sum_i\sij =1 \forall j \nonumber\\
&\sum_j\sij =1 \forall i \nonumber\\
&\sij\in\{0,1\} \forall i,j \label{eq:IP}.
\ean
\end{minipage}%
\begin{minipage}{.5\linewidth}
\ban
\max & \sum\lij \qij\sij  \nonumber\\
\text{subject to: } & \sum_i\sij =1 \forall j \nonumber\\
&\sum_j\sij =1 \forall i \nonumber\\
&\sij\geq 0 \forall i,j \label{eq:LP_relaxation}.
\ean
\end{minipage}
\\

It can be proved that the optimal solution of the LP relaxation (\ref{eq:LP_relaxation}) is identical to the optimal solution of the original integer program (\ref{eq:IP}) \cite{schrijver_combinatorial_book}. The dual of the LP (\ref{eq:LP_relaxation}) is the following.
\ban
\min & \sum_i w_i +\sum_j \widetilde{w}_j \nonumber\\
\text{subject to: } & w_i + \widetilde{w}_j \geq \qij  \forall i,j \label{eq:dual}
\ean
For any perfect matching $\pi$ and its corresponding schedule $\sij$, and for any dual feasible $w_i,\widetilde{w}_j$, we have that $\sum_i q_{i\pi(i)}= \sum\lij \qij\sij \leq \sum_i w_i +\sum_j \widetilde{w}_j$.
Suppose $\sij^*$ is an optimal solution of \ref{eq:IP} and corresponds to a permutation $\pi^*$, and suppose $w_i^*,\widetilde{w}_j^*$ is an optimal solution of \ref{eq:dual}. Then, from strong duality, we have that
\ban
\sum_i q_{i\pi^*(i)}= \sum\lij \qij\sij^* = \sum_i w_i^* +\sum_j \widetilde{w}_j^*.\label{eq:strong_duality}
\ean
Moreover, any $\pi^*$ and $w_i^*,\widetilde{w}_j^*$ that satisfy (\ref{eq:strong_duality}) are optimal solutions for problems (\ref{eq:IP})( or (\ref{eq:LP_relaxation})) and (\ref{eq:dual}) respectively. This means that any optimal perfect matching consists of only links $(i,j)$ such that $\qij=w_i+\widetilde{w}_j$. This property is also called  complementary slackness. The Hungarian assignment algorithm  for solving the MaxWeight matching problem is based on this property. The cone $\cK$, has the special property that if $w_i,\widetilde{w}_j$ is the optimal solution, then for any $(i,j)$, we have $\qij=w_i+\widetilde{w}_j$ and so any perfect matching is an optimal matching and all perfect matchings have the same weight.

The fact that all perfect matchings have same weight when $\vq \in \cK$ can be used to give an alternate proof of Lemma \ref{lem:cone_q_avg}.
The average weight of all the $n!$ perfect matchings is $\frac{1}{n} \sum_{i',j'} q_{i'j'}$. Now consider the matchings that contain the edge $ij$. There are $(n-1)!$ such matchings. The total weight of all these matchings is $(n-1)!q_{ij}+\sum_{i'\neq i}\sum_{j'\neq j} (n-2)! q_{i'j'}$, because every edge $i'j'$ appears in $(n-2)!$ of these $(n-1)!$ matchings. Since all the matchings have same weight, equating the average weight of these $(n-1)!$ matchings to the average of all the matchings, we have
\ba
q_{ij}+\frac{1}{n-1}\sum_{i'\neq i}\sum_{j'\neq j}  q_{i'j'}=&\frac{1}{n} \sum_{i',j'} q_{i'j'} \\
q_{ij}+\frac{1}{n-1}\left(\sum_{i'}\sum_{j'}  q_{i'j'} -\sum_{j'=1}^n q_{ij'} - \sum_{i'=1}^n q_{i'j}  +q_{ij} \right)=&\frac{1}{n} \sum_{i',j'} q_{i'j'} \\
q_{ij}\left(1+\frac{1}{n-1}\right) -\frac{1}{n-1}\left( \sum_{j'=1}^n q_{ij'}+ \sum_{i'=1}^n q_{i'j} \right)= & \sum_{i',j'} q_{i'j'}\left(\frac{1}{n}-\frac{1}{n-1}\right)\\
n q_{ij} - \left( \sum_{j'=1}^n q_{ij'}+ \sum_{i'=1}^n q_{i'j} \right)= & -\sum_{i',j'} q_{i'j'}\left(\frac{1}{n}\right) \quad \text{when } n>1,
\ea
which gives Lemma \ref{lem:cone_q_avg}.

\section{Asymptotically tight Upper and Lower bounds under MaxWeight policy}\label{sec:MaxWt_UB}

In the previous section, we have shown that the queue length vector collapses within the cone \cK in the steady state. We will use this result to obtain lower and upper bounds on the average queue lengths under MaxWeight algorithm. The lower and upper bounds differ only in $o(1/\epsilon)$ and so match in the heavy traffic limit.

We will obtain these bounds by equating the drift of certain carefully chosen functions equal to zero in steady-state. We first define a few Lyapunov-type functions and their drifts, in addition to the already defined $\Vq =\|\vq\|^2$. 
The following lemma states that all these Lyapunov functions have finite expectations in steady state. 
\begin{gather*}
\Viq \triangleq \sum_i \left(\sum_j\qij \right)^2 \quad
\Voq \triangleq \sum_j \left(\sum_i \qij \right)^2 \quad
\Vtq \triangleq  \left(\sum\lij \qij \right)^2 \\
\begin{aligned}
\Delta \Viq \triangleq & [\Vi(\vq(t+1))-\Vi(\vq(t))]\>\mathcal{I}(\vq(t)=\vq)  \\
\Delta \Voq \triangleq & [\Vo(\vq(t+1))-\Vo(\vq(t))]\>\mathcal{I}(\vq(t)=\vq) \nonumber\\
\Delta \Vtq \triangleq & [\Vt(\vq(t+1))-\Vt(\vq(t))]\>\mathcal{I}(\vq(t)=\vq) \\
\end{aligned}
\end{gather*}

\begin{lemma}\label{lem:V_bounded}
Consider the switch under MaxWeight scheduling algorithm.
For any arrival rate vector \vlam in the interior of the capacity region $\vlam \in \text{int}(\cC)$,
the steady state means $\bE[\V(\bvq)]$, $\bE[\Vi(\bvq)]$, $\bE[\Vo(\bvq)]$ and $\bE[\Vt(\bvq)]$ are finite.
\end{lemma}

The lemma is proved in Appendix \ref{sec:V_bounded_proof}.
We will now state and prove the main result of this paper.

\begin{theorem}\label{thm:UB_maximalface}
Consider a set of switch systems under MaxWeight scheduling algorithm, with the arrival processes
$\va\se(t)$ described in Section \ref{sub:model}, parameterized by $0<\epsilon<1$, such that the mean arrival rate vector is $\vlam^{\epsilon}=(1-\epsilon)\vnu$ for some $\vnu \in \cF$ such that $\numin\triangleq \min\lij \nu\lij >0$.
 The load is then $\rho=(1-\epsilon)$.
Let the variance of the arrival process be $\left(\vsig^{(\epsilon)}\right)^2$.
The queue length process $\vq^{(\epsilon)}(t)$ for each system converges in distribution to the steady state random vector $\bvq^{(\epsilon)}$.
For each system with $0< \epsilon \leq \numin/2\|\vnu\|$, the steady state average queue length satisfies
\ba
  \left(1-\frac{1}{2n}\right)\frac{\left\|\vsig\se\right\|^2}{\epsilon} -\Bi(\epsilon,n) \leq \bE \left [ \sum\lij\bqij\se \right ] \leq \left(1-\frac{1}{2n}\right)\frac{\left\|\vsig\se\right\|^2}{\epsilon} +\Bii(\epsilon,n)
\ea
where
\ba
\Bi(\epsilon,n) =
-\frac{n\epsilon}{2}
+n
+3n^{\left(2-\frac{1}{r}\right)}
\epsilon ^{\left(-\frac{1}{r}\right)}  M_r\se
\quad \text{and}\quad
\Bii(\epsilon,n) =
\frac{n(1+\epsilon)}{2}
+ 2n^{\left(2-\frac{1}{r}\right)}\epsilon^{\left(-\frac{1}{r}\right)}
M_r\se
\ea
for any $r\in\{2,3,\ldots\}$. The terms $\Bi(\epsilon,n)$ and $\Bii(\epsilon,n)$
are both
 $o\left(\frac{1}{\epsilon}\right)$, i.e., $\lim_{\epsilon\downarrow 0} \epsilon \Bi(\epsilon,n) = 0$ and $\lim_{\epsilon\downarrow 0} \epsilon \Bii(\epsilon,n) = 0$. Therefore, in the heavy traffic limit as $\epsilon\downarrow 0$ which means as the mean arrival rate $\vlam^{\epsilon}\rightarrow \frac{1}{n}\vone$,
if $\left(\vsig^{(\epsilon)}\right)^2 \rightarrow \vsig^2$, we have
\ba
\lim_{\epsilon\downarrow 0}  \epsilon \bE \left [ \sum\lij\bqij\se \right ] =  \left(1-\frac{1}{2n}\right)\left\|\vsig\right\|^2
\ea
\end{theorem}
\begin{proof}
Fix an $0< \epsilon \leq \numin/2\|\vnu\|$ and we consider the system with index $\epsilon$. For simplicity of notation,  we again skip the superscript $(\epsilon)$ in this proof and use $\bvq$ to denote the steady state queue length vector.
We will use $\bva$ to denote the arrival vector in steady state, which is identically distributed to the  random vector $\va(t)$ for any time $t$.
We will use $\vs(\bvq)$ and $\vu(\bvq)$ to denote the schedule and unused service to show their dependence on the queue lengths. 
We will use $\bvq^+$ to denote $\bvq+\bva-\vs(\bvq)+\vu(\bvq)$, which is the queue lengths vector at time $t+1$ if it was $\bvq$ at time $t$. Clearly, $\bvq^+$ and $\bvq$ have the same distribution.

Define a new function $\Vuq$ and its drift as follows.
\ba
\Vuq=& \Viq+\Voq-\frac{1}{n}\Vtq\\
 = & \sum_i \left(\sum_j\qij \right)^2 +\sum_j \left(\sum_i \qij \right)^2 - \frac{1}{n} \left(\sum\lij \qij \right)^2
\\
\Delta \Vuq \triangleq & [\Vu(\vq(t+1))-\Vu(\vq(t))]\>\mathcal{I}(\vq(t)=\vq) \\
 = & \Delta \Viq+\Delta \Voq-\frac{1}{n}\Delta \Vtq
\ea
Since $-\frac{1}{n}\Vtq \leq \Vuq \leq \Viq+\Voq$, the steady state mean $\bE[\Vu(\bvq)]$ is finite from Lemma \ref{lem:V_bounded}. Therefore, the mean drift of $\Vu(.)$ in steady state is zero, i.e.,
\ba
\bE[\Delta \Vu(\bvq)] = \bE[[\Vu(\vq(t+1)\hspace{-.2em})\hspace{-.2em}-\hspace{-.2em} \Vu(\vq(t)\hspace{-.2em})]\>\mathcal{I}(\vq(t)\hspace{-.2em}=\hspace{-.2em}\bvq)] = \bE[\Vu(\bvq^+\hspace{-.2em})] \hspace{-.2em}-\hspace{-.2em} \bE[\Vu(\bvq)] = \bE[\Vu(\bvq)] \hspace{-.2em}-\hspace{-.2em} \bE[\Vu(\bvq)] = 0
\ea
\ban
0 = & \bE[\Delta \Vu(\bvq)] \nonumber\\
= & \bE[\Delta \Vi(\bvq)] +\bE[\Delta \Vo(\bvq)]-\frac{1}{n}\bE[\Delta \Vt(\bvq)] \label{eq:delta_Vu}
\ean
Expanding the drift of $\Vi(.)$, we get
\ba
\lefteqn{\bE[\Delta \Vi(\bvq)]}\\
= & \bE[\Vi(\bvq+\bva-\vs(\bvq)+\vu(\bvq))-\Vi(\bvq)]\\
= & \bE\left[\sum_i \left(\sum_j (\bqij+\baij-\sij(\bvq)+\uij(\bvq)) \right)^2 -\sum_i \left(\sum_j\bqij \right)^2 \right]\\
%
%
= & \bE\left[\sum_i \left(\sum_j(\baij-\sij(\bvq))\right)^2  +2\sum_i\left(\sum_j(\bqij+\baij-\sij(\bvq))\right)\left(\sum_{j'} u_{ij'}(\bvq)\right)\right] \\
&+\bE \left[+ \sum_i \left(\sum_j\uij(\bvq)  \right)^2+2\sum_i\left(\sum_j\bqij\right)\left(\sum_{j'}(\overline{a}_{ij'}-s_{ij'}(\bvq))\right) \right] \\
= & \bE\left[\sum_i \left(\sum_j(\baij-\sij(\bvq))\right)^2 - \sum_i \left(\sum_j\uij(\bvq)  \right)^2 +2\sum_i\left(\sum_j\bqij^+\right)\left(\sum_{j'} u_{ij'}(\bvq)\right)\right] \\
&+2\bE \left[\sum_i\left(\sum_j\bqij\right)\left(\sum_{j'}(\overline{a}_{ij'}-s_{ij'}(\bvq))\right) \right].
\ea
Similarly expanding drifts of $\Vo(.)$ and $\Vt(.)$ and substituting in (\ref{eq:delta_Vu}), we get the following expression. Since this is a lengthy equation, we split into various terms which we denote by \Ti,\Tii,\Tiii and \Tiv. For simplicity of notation, we suppress all the dependencies in terms of \bvq, \bva, \vs(\bvq) and \vu(\bvq).
\ban
\Ti=\Tii+\Tiii+\Tiv \label{eq:T1234}
\ean
where
\ba
\Ti=&
2\bE \left[\sum_i\left(\sum_j\bqij\right)\left(\sum_{j'}(s_{ij'}(\bvq)-\overline{a}_{ij'})\right) \right]
+2\bE \left[\sum_j\left(\sum_i\bqij\right)\left(\sum_{i'}(s_{i'j}(\bvq)-\overline{a}_{i'j})\right) \right]\\
 & - \frac{2}{n} \bE\left[\left(\sum\lij\bqij\right)\left(\sum_{i'j'}(s_{i'j'}(\bvq)-\overline{a}_{i'j'})\right) \right]\\
\Tii=&
\bE\left[\sum_i \left(\sum_j(\baij-\sij(\bvq))\right)^2\right]
+\bE\left[\sum_j \left(\sum_i(\baij-\sij(\bvq))\right)^2\right]
-\frac{1}{n} \bE\left[\left(\sum\lij(\baij-\sij(\bvq))\right)^2\right]\\
\Tiii=&
-\bE\left[\sum_i \left(\sum_j\uij(\bvq)  \right)^2\right]
-\bE\left[\sum_j \left(\sum_i\uij(\bvq)  \right)^2\right]
+\frac{1}{n} \bE\left[\left(\sum\lij\uij(\bvq)  \right)^2\right] \\
\Tiv = &
2\bE\left[\sum_i\left(\sum_j\bqij^+\right)\left(\sum_{j'} u_{ij'}(\bvq)\right)\right]
+2\bE\left[\sum_j\left(\sum_i\bqij^+\right)\left(\sum_{i'} u_{i'j}(\bvq)\right)\right]\\
& -\frac{2}{n}\bE\left[\left(\sum\lij\bqij^+\right)\left(\sum_{i'j'} u_{i'j'}(\bvq)\right)\right]
\ea

We will now bound each of the four terms.
The schedule in each time slot is maximal (\ref{eq:maximal_sched}) and so $\sum_is\lij=1, \sum_js\lij=1$ and $\sum\lij s\lij=n$. Noting that the arrivals are independent of queue lengths, we can simplify the term \Ti as follows.
\ban
\Ti=&
2\bE \left[\sum_i\left(\sum_j\bqij\right)\left(1-\sum_{j'}\lambda_{ij'}\right) \right]
+2\bE \left[\sum_j\left(\sum_i\bqij\right)\left(1-\sum_{i'}\lambda_{i'j}\right) \right] \nonumber\\
 & - \frac{2}{n} \bE\left[\left(\sum\lij\bqij\right)\left(n-\sum_{i'j'}\lambda_{i'j'}\right) \right] \nonumber\\
\stackrel{(a)}{=}&2\bE \left[\sum_i\epsilon \left(\sum_j\bqij\right) \right]
+2\bE \left[\sum_j\epsilon \left(\sum_i\bqij\right) \right] - \frac{2}{n} \bE\left[n\epsilon \left(\sum\lij\bqij\right)\right] \nonumber\\
=&2\epsilon \bE \left[\sum\lij\bqij \right], \nonumber
\ean
 where (a) follows from the fact that $\sum_{j}\lambda_{ij} =1-\epsilon$ and $\sum_{i}\lambda_{ij} =1-\epsilon$ since $\vlam^{\epsilon}=(1-\epsilon)\vnu$ and $\vnu \in \cF$.

Thus, from (\ref{eq:T1234}), we have
\ban
2\epsilon \bE \left[\sum\lij\bqij \right] = \Tii+\Tiii+\Tiv.
\label{eq:LHST234}
\ean
Now the rest of the proof involves bounding the term $\Tii, \Tiii$ and $\Tiv$. We start with the term $\Tii$. Consider the first term of $\Tii$. Again noting that the schedules are maximal (\ref{eq:maximal_sched}), we get
\ba
\bE\left[\sum_i \left(\sum_j(\baij-\sij(\bvq))\right)^2\right] = & \sum_i \bE\left[ \left(\sum_j\baij-1\right)^2\right] \\
= & \sum_i \bE\left[ \left(\sum_j\baij-(1-\epsilon)-\epsilon\right)^2\right] \\
= & \sum_i \bE\left[ \left(\sum_j\baij-(1-\epsilon)\right)^2\right]+\sum_i\epsilon^2-\sum_i2\epsilon\bE\left[ \left(\sum_j\baij-(1-\epsilon)\right)\right] \\
\stackrel{(a)}{=}& n\epsilon^2 +\sum_i \text{Var}\left(\sum_j\baij\right)\\
\stackrel{(b)}{=}& n\epsilon^2 +\sum\lij \sigma^2\lij\\
=& n\epsilon^2+\|\vsig\|^2,
\ea
where (a) is true because $\bE[\sum_j\baij]=(1-\epsilon)$; (b) follows from the independence of the arrival processes across ports.
Similarly, we can show that the second term in \Tii evaluates to\\ $\bE\left[\sum_i \left(\sum_j(\baij-\sij(\bvq))\right)^2\right] = n\epsilon^2+\|\vsig\|^2$. The last term can likewise be evaluated as follows.
\ba
\frac{1}{n}\bE\left[ \left(\sum\lij(\baij-\sij(\bvq))\right)^2\right] = & \frac{1}{n} \bE\left[ \left(\sum\lij\baij-n\right)^2\right] \\
= & \frac{1}{n} \bE\left[ \left(\sum\lij\baij-n(1-\epsilon)-n\epsilon\right)^2\right] \\
= & \frac{1}{n}\bE\left[ \left(\sum\lij\baij-n(1-\epsilon)\right)^2\right]+n\epsilon^2-2\epsilon\bE\left[ \left(\sum\lij\baij-n(1-\epsilon)\right)\right] \\
=& n\epsilon^2 +\frac{1}{n} \text{Var}\left(\sum\lij\baij\right)\\
=& n\epsilon^2 +\frac{1}{n}\sum\lij \sigma^2\lij\\
=& n\epsilon^2+\frac{1}{n}\|\vsig\|^2,
\ea
Putting all the terms of \Tii together, we get
\ban
\Tii  =& n\epsilon^2+\left(2-\frac{1}{n}\right)\|\vsig\|^2 \label{eq:T2}. 
\ean

Since $\sum\lij\qij \in\bZ_+$, we have $\sum\lij\qij \leq (\sum\lij\qij)^2$. Using the fact that $\bE \left [ (\sum\lij\bqij)^2 \right ]$ is finite from Lemma \ref{lem:V_bounded}, we have that $\bE \left [ \sum\lij\bqij \right ]$ is finite and so its drift is zero in steady state. Thus, we get
\ban
0=& \bE\left [\left[ \sum\lij\qij(t+1)-\sum\lij\qij(t)\right]\>\mathcal{I}(\vq(t)=\bvq)\right ] \nonumber\\
=& \bE\left [ \sum\lij\overline{a}\lij-\sum\lij s\lij(\bvq)+ \sum\lij u\lij(\bvq)\right ] \nonumber\\
\bE\left [  \sum\lij u\lij(\bvq)\right ]=& n- n(1-\epsilon) \nonumber\\
=& n \epsilon \label{eq:unepsilon}
\ean
We will now bound the term \Tiii. Since $u\lij(t)\leq s\lij(t)$, we have $\sum_iu\lij\leq1, \sum_ju\lij\leq1$ and $\sum\lij u\lij\leq n$. Therefore,
\begin{alignat}{3}
-\bE\left[\sum_i \left(\sum_j\uij(\bvq)  \right)^2\right]
-\bE\left[\sum_j \left(\sum_i\uij(\bvq)  \right)^2\right]
 & \leq \Tiii \leq &&
\frac{1}{n} \bE\left[\left(\sum\lij\uij(\bvq)  \right)^2\right]
 \nonumber\\
-\bE\left[\sum_i \left(\sum_j\uij(\bvq)  \right)\right]
-\bE\left[\sum_j \left(\sum_i\uij(\bvq)  \right)\right]
 & \leq \Tiii \leq &&
 \frac{1}{n} \bE\left[n\left(\sum\lij\uij(\bvq)  \right)\right]
 \nonumber\\
-2n \epsilon & \leq \Tiii \leq &&n \epsilon \label{eq:T3}
\end{alignat}

We now consider the term \Tiv. It can be rewritten as follows, and can be split into two parts, one each corresponding to $\bvq^+\para$ and $\bvq^+\per$, where $\bvq^+\para$ means $\left(\bvq^+\right)\para$ and similarly $\bvq^+\per$.
\ba
\Tiv
=& 2\bE\left[\sum\lij \uij(\bvq) \left( \sum_{j'}\bq_{ij'}^+ + \sum_{i'}\bq_{i'j}^+ -\frac{1}{n} \sum_{i'j'}\bq_{i'j'}^+ \right)\right]\\
=& 2\bE\left[\sum\lij \uij(\bvq) \left( \sum_{j'}\bq_{\parallel ij'}^+ + \sum_{i'}\bq_{\parallel i'j}^+ -\frac{1}{n} \sum_{i'j'}\bq_{\parallel i'j'}^+ \right)\right]\\
&+ 2\bE\left[\sum\lij \uij(\bvq) \left( \sum_{j'}\bq_{\perp ij'}^+ + \sum_{i'}\bq_{\perp i'j}^+ -\frac{1}{n} \sum_{i'j'}\bq_{\perp i'j'}^+ \right)\right]
\ea
Since the vector $\bvq\para^+$ is in cone \cK by definition, Lemma \ref{lem:cone_q_avg} is applicable.
Recall that when $\uij(t)=1$, $\qij(t+1)=0$. Thus, when $\uij(\bvq)=1$, we have
\ban
\bqij^+=&0 \nonumber\\
\bq_{\parallel ij}^+ =&- \bq_{\perp ij}^+ \nonumber\\
\frac{1}{n}\sum_{j'=1}^n\bq_{\parallel ij'}^+ + \frac{1}{n}\sum_{i'=1}^n \bq_{\parallel i'j}^+ -\frac{1}{n^2} \sum_{i'=1}^n\sum_{j'=1}^n \bq_{\parallel i'j'}^+ =&- \bq_{\perp ij}^+ \nonumber
\ean
Therefore, we get
\ba
\uij(\bvq) \left( \sum_{j'}\bq_{\parallel ij'}^+ + \sum_{i'}\bq_{\parallel i'j}^+ -\frac{1}{n} \sum_{i'j'}\bq_{\parallel i'j'}^+ \right) = -n\uij(\bvq) \bq_{\perp ij}^+
\ea
and the term \Tiv reduces to
\ban
\Tiv
=& 2\bE\left[\sum\lij \uij(\bvq) \left(-n\bq_{\perp ij}^+ + \sum_{j'}\bq_{\perp ij'}^+ + \sum_{i'}\bq_{\perp i'j}^+ -\frac{1}{n} \sum_{i'j'}\bq_{\perp i'j'}^+ \right)\right] \nonumber\\
=& 2\bE\left[\La\vu(\bvq),-n\bvq\per^+  +\sum_i\La\bvq\per^+,\vei\Ra \vei +\sum_j\La\bvq\per^+,\vetj\Ra \vetj -\frac{1}{n}\La\bvq\per^+,\vone\Ra\vone  \Ra\right] \label{eq:T4_motivation}.
\ean
Term \Tiv is a critical term to bound and our choice of the Lyapunov function $\Vu(.)$ is motivated primarily to obtain  
\eqref{eq:T4_motivation}. We explain the motivation in detail at the end of this section.
From state space collapse, we know that $\bvq\per^+$  is bounded. We will now use this result to show that \Tiv is $o(\epsilon)$.
Since $\bvq\per^+ \in \cKo$ and $\vei, \vetj, \vone \in \cK$ for all $i,j$, we have that $\La\bvq\per^+,\vei\Ra \leq 0$, $\La\bvq\per^+,\vetj\Ra \leq 0$ and $\La\bvq\per^+,\vone\Ra \leq 0$. Moreover all components of $\vu,\vei,\vetj$ and $\vone$ take values $0$ and $1$. Therefore,
\ba
\Tiv \leq &
 2\bE\left[\La\vu(\bvq),-n\bvq\per^+  -\frac{1}{n}\La\bvq\per^+,\vone\Ra\vone  \Ra\right] \nonumber\\
\stackrel{(a)}{\leq}&
 2\left(\bE\left[\|\vu(\bvq)\|_{\rt}^{\rt}\right] \right)^{\frac{1}{\rt}}
 \left(\bE\left[\left\|-n\bvq\per^+  -\frac{1}{n}\La\bvq\per^+,\vone\Ra\vone \right\|_r^r \right]\right)^{\frac{1}{r}}
 \nonumber\\
\stackrel{(b)}{\leq}&
 2\left(n\epsilon \right)^{\frac{1}{\rt}}
 \left(\bE\left[\left(n\|\bvq\per^+\|_r +\frac{1}{n}\left|\La\bvq\per^+,\vone\Ra\right| \|\vone \|_r \right)^r \right]\right)^{\frac{1}{r}}
 \nonumber\\
\stackrel{(c)}{\leq}&
 2\left(n\epsilon \right)^{\frac{1}{\rt}}
 \left(\bE\left[\left(n\|\bvq\per^+\|_r +\frac{1}{n} \|\bvq\per^+\|_r \|\vone \|_{\rt}  \|\vone \|_r \right)^r \right]\right)^{\frac{1}{r}}
 \nonumber\\
\stackrel{(d)}{=}&
 2\left(n\epsilon \right)^{\frac{1}{\rt}}
 \left(\bE\left[\left(n\|\bvq\per^+\|_r +\frac{(n^2)^{\left(\frac{1}{\rt}+\frac{1}{r}\right)}}{n} \|\bvq\per^+\|_r  \right)^r \right]\right)^{\frac{1}{r}}
 \nonumber\\
\stackrel{(e)}{=}&
 4n^{\left(1+\frac{1}{\rt}\right)}\epsilon^{\frac{1}{\rt}}
 \left(\bE\left[
 \|\bvq\per^+\|_r ^r \right]\right)^{\frac{1}{r}}
 \nonumber\\
\stackrel{(f)}{\leq}&
 4n^{\left(1+\frac{1}{\rt}\right)}\epsilon^{\frac{1}{\rt}}
 \left(\bE\left[
 \|\bvq\per^+\|_2 ^r \right]\right)^{\frac{1}{r}} \qquad \text{ for } r\geq 2
 \nonumber\\
\stackrel{(g)}{\leq}&
 4n^{\left(1+\frac{1}{\rt}\right)}\epsilon^{\frac{1}{\rt}}
M_r\se \qquad\qquad\qquad\ \   \text{ for } r\geq 2
 \nonumber\\
\leq&
 4n^{\left(2-\frac{1}{r}\right)}\epsilon^{\left(1-\frac{1}{r}\right)}
M_r\se \qquad\qquad \quad \ \ \text{ for } r\geq 2
 \nonumber
\ea
where $\|\vx\|_r$ denotes the $\ell_r$ norm of a vector \vx, and $r,\rt \in (1,\infty)$ satisfy $1/r+1/\rt=1$.
Inequality (a) follows from the
H\"older's inequality for random vectors. 
Cauchy-Schwartz inequality (which is a special case of H\"older's inequality) may also be used to obtain the same bound in heavy traffic limit. However, in the non-heavy traffic limit, H\"older's inequality gives a tighter bound.
Since $\uij\in\{0,1\}$, from (\ref{eq:unepsilon}), we have $\bE\left[\|\vu(\bvq)\|_{\rt}^{\rt}\right]= \bE\left[\sum\lij (u\lij(\bvq))^{\rt}\right] =\bE\left[\sum\lij u\lij(\bvq)\right]=n\epsilon$. This fact along with using triangle inequality on the second term gives (b). Inequality (c) again follows using H\"older's inequality for vectors. The $\ell_r$ norm of vector \vone is $\|\vone \|_r=n^{2/r}$, this gives (d). Since $\frac{1}{\rt}+\frac{1}{r}=1$, we have (e). For any vector $\vx$, if $0<r<r'$, we have $\|\vx\|_{r'}\leq \|\vx\|_{r}$, 
and this gives (f) and (g) follows from state space collapse in Proposition \ref{prop:SSC}. The last inequality follows from $1/r+1/\rt=1$.
 Similarly, we can lower bound \Tiv as follows.
%

\ban
 \Tiv \geq &
2\bE\left[\La\vu(\bvq),-n\bvq\per^+  +\sum_i\La\bvq\per^+,\vei\Ra \vei +\sum_j\La\bvq\per^+,\vetj\Ra \vetj \Ra\right] \nonumber\\
\geq &
-2
\left(\bE\left[\|\vu(\bvq)\|_{\rt}^{\rt}\right] \right)^{\frac{1}{\rt}} \left(\bE\left[\left\|-n\bvq\per^+  +\sum_i\La\bvq\per^+,\vei\Ra \vei +\sum_j\La\bvq\per^+,\vetj\Ra \vetj \right\|^r_r\right]\right)^{\frac{1}{r}}
  \nonumber\\
\geq &
-2
\left(n\epsilon \right)^{\frac{1}{\rt}} \left(\bE\left[\left( \left\|n\bvq\per^+ \right\|_r +\left\|\sum_i \La\bvq\per^+,\vei\Ra \vei \right\|_r +\left\|\sum_j \La\bvq\per^+,\vetj\Ra \vetj \right\|_r\right)^r\right]\right)^{\frac{1}{r}}.\label{eq:T4LB_temp}\ean
Let's now focus on the middle term in the expectation above. From the definition of $\vei$, we have
\ba
\left\|\sum_i \La\bvq\per^+,\vei\Ra \vei \right\|_r = & \left(\sum_i n \left|\La\bvq\per^+,\vei\Ra\right|^r \right)^{\frac{1}{r}} \\
= & \left(\sum_i n \left(\sum_j q\perij^+\right)^r \right)^{\frac{1}{r}} \\
\stackrel{(a)}{\leq}& \left(\sum_i n^r \sum_j \left( q\perij^+\right)^r \right)^{\frac{1}{r}} \\
=& n \left\|\bvq\per^+ \right\|_r.
\ea
For any $(x_1,\ldots,x_n)\in\bR^n$ and $r\geq 1$, from Jensen's inequality, we have $\left(\frac{\sum_ix_i}{n}\right)^r \leq \frac{\sum_ix_i^r}{n}$. This gives inequality (a) above. We have a similar bound for the last term in expectation in (\ref{eq:T4LB_temp}). Using both these bounds, the lower bound on \Tiv becomes,
\ba
\Tiv \geq &
-6n^{\left(1+\frac{1}{\rt}\right)}
\epsilon ^{\frac{1}{\rt}} \left(\bE\left[\left( \left\|\bvq\per^+ \right\|_r  \right)^r\right]\right)^{\frac{1}{r}}
  \nonumber\\
\geq&
-6n^{\left(1+\frac{1}{\rt}\right)}
\epsilon ^{\frac{1}{\rt}} \left(\bE\left[\left( \left\|\bvq\per^+ \right\|_2  \right)^r\right]\right)^{\frac{1}{r}} \qquad \text{ for } r\geq 2
  \nonumber\\
\geq&
-6n^{\left(1+\frac{1}{\rt}\right)}
\epsilon ^{\frac{1}{\rt}}  M_r\se \qquad \text{ for } r\geq 2
  \nonumber\\
\geq&
-6n^{\left(2-\frac{1}{r}\right)}
\epsilon ^{\left(1-\frac{1}{r}\right)}  M_r\se \qquad \text{ for } r\geq 2
  \nonumber
  \ea
%
Combining the lower and upper bounds on $\Tiv$, for $r\geq 2$, we have
\ban
-6n^{\left(2-\frac{1}{r}\right)}
\epsilon ^{\left(1-\frac{1}{r}\right)}  M_r\se
\leq \Tiv \leq
 4n^{\left(2-\frac{1}{r}\right)}\epsilon^{\left(1-\frac{1}{r}\right)}
M_r\se. \label{eq:T4}
\ean
Using (\ref{eq:T2}),(\ref{eq:T3}) and (\ref{eq:T4}) to bound (\ref{eq:LHST234}) and reintroducing the superscript $\se$, we get the theorem.
\end{proof}

We will now present the motivation for the choice of the function $\Vu(.)$. First consider a discrete-time single server (G/G/1) queue, $q(t)$ that evolves according to $q(t+1)=q(t)+a(t)-s(t)+u(t)$. The queue $\phi(t)$ in Section \ref{sec:ULB} is an example.
Similar to \eqref{eq:LB_single_server2}, we can obtain tight lower and upper bounds on mean queue length in steady state by setting the drift of $\bE[\bq^2]$ to be zero in steady state, i.e, $\bE[\bq^2(t+1)]=\bE[\bq^2(t)]$. Such a bound is called Kingman bound. See \cite[Section 10.1]{srikantleibook}.
When expanded, this equation again gives four terms, similar to  the terms $\Ti,\Tii,\Tiii$ and $\Tiv$. The fourth term \Tiv then is $u(\bq)\bq^+$, which is zero from the definition of unused service. This is an important step in obtaining tight bounds.

Next, consider a load balancing system, similar to a super market checkout lanes. There are $n$ servers with a separate queue for each server. Whenever a user arrives into the system, (s)he picks one of the servers and joins the corresponding queue. We consider `Join the shortest queue'(JSQ) policy, in which each user joins the queue with the shortest length. Ties are broken uniformly at random. The queue length at server $i$ then evolves according to $q_i(t+1)=q_i(t)+a_i(t)-s_i(t)+u_i(t)$. It was shown in \cite{erysri-heavytraffic} that the JSQ policy has minimum steady state sum queue lengths in heavy traffic. This was done by first showing that the queue lengths collapse to a single dimension where they are all equal. A tight upper bound is then obtained by setting the drift of the quadratic function $\bE[(\sum_i \bq_i)^2]$ to be zero in steady state. When this equation is expanded, we again have four terms and the fourth one being of the form $(\sum_iu_i(\bq))(\sum_{i'}\bq_{i'}^+)$. This is not zero in general because of the cross terms. However, when the state is such that all the queue lengths are equal, this term is zero. This is easy to see by considering the term $u_i(\bq)(\sum_{i'}\bq_{i'}^+)$. When $u_i=1$, we have that $\bq_{i}^+=0$ and when all the queues are equal, for any $i'$, $\bq_{i'}^+=0$.

Therefore, in all these systems, when using a quadratic Lyapunov function, the fourth term \Tiv is the most important and challenging one to bound correctly. Usually, it should be zero if state space collapse is such that $\bvq\per^+=\bf{0}$.
However, for the switch system, if we use Lyapunov functions $\Vi(.)$ or $\Vo(.)$ or $\Vt(.)$ or $\Vi(.)+\Vo(.)$, we do not have the property that $\Tiv=0$ when $\bvq\per^+=\bf{0}$. Armed with Lemma \ref{lem:cone_q_avg}, we add the additional $-\Vt(.)$ to $\Vi(.)+\Vo(.)$ to obtain the Lyapunov function $\Vu(.)$.
We have shown in \eqref{eq:T4_motivation} that \Tiv is zero when $\bvq \in\cK$ (since $\bvq\per^+=\bf{0}$).
The key idea in our upper bound proof is the choice of the function $\Vu(.)$. Essentially, we picked the function $\Vu(.)$ so that it matches with the geometry of the cone \cK in the sense that if the queue length vector is in the cone $\cK$, the fourth term \Tiv is zero.

\section{Uniformly loaded switch under Bernoulli traffic}
In this section, we consider the switch system when all the ports have Bernoulli traffic with same arrival rate. The lower and upper bound expressions then have much simple form.
More precisely, for the system with index $\epsilon$ , for every input-output pair (i,j), the arrival process  $\aij\se(t)$ is a Bernoulli process with rate $\lambda \lij = (1-\epsilon)/n$. In other words, the rate vector approaches the vector $\vnu=\vone/n\in\cF$ on the face \cF as $\epsilon\rightarrow 0$. Then, clearly the variance vector for the system with index $\epsilon$ is $\left(\vsig\se\right)^2=\frac{1-\epsilon}{n}(1-\frac{1-\epsilon}{n})\vone$ with $\left\|\vsig\se\right\|^2=(1-\epsilon)(n-(1-\epsilon))$ and it converges to $\vsig^2=\frac{n-1}{n^2}\vone$. Moreover, $a_{\max}=1$ and $\numin=\frac{1}{n}$. Using these values, we can restate Propositions \ref{prop:Universal_LB} and \ref{prop:SSC}, and Theorem \ref{thm:UB_maximalface} as follows:
\begin{theorem}\label{thm:uniform_arrivals}
Consider a set of switch systems with the Bernoulli arrival processes
$\va\se(t)$ parameterized by $0<\epsilon<1$, such that the mean arrival rate vector is $\vlam^{\epsilon}=\frac{1-\epsilon}{n}\vone$.
Fix a scheduling policy under which the switch system is stable for any $0<\epsilon<1$.
Let $\vq^{(\epsilon)}(t)$ denote the queue lengths process under this policy for each system. Suppose that this process converges in distribution to a steady state random vector $\bvq^{(\epsilon)}$.
Then, for each of these systems, the average queue length is lower bounded by
\ba
\bE \left [\sum\lij \bqij \se \right ] \geq \frac{(1-\epsilon)^2}{2\epsilon}(n-1)
\ea
Therefore, in the heavy-traffic limit as $\epsilon \downarrow 0$, we have
\ba
\liminf_{\epsilon \downarrow 0} \epsilon \bE \left [\sum\lij \bqij \se \right ] \geq \frac{n-1}{2}
\ea

Now consider the same switch systems operating under the MaxWeight scheduling algorithm.
The queue length process $\vq^{(\epsilon)}(t)$ of each system is positive recurrent and so converges to a steady state random vector in distribution $\bvq^{(\epsilon)}$.
Then,
for each system with $0< \epsilon \leq 1/2n$, the steady state queue lengths vector collapses into the cone \cK in the sense that it satisfies
\ba
\bE \left [\| \bvq\per \se \|^r\right ] \leq (\widetilde{M}_r)^r \>\> \forall r \in\{1,2,\ldots\}, \quad \text{where}\quad
\widetilde{M}_r=
(2\sqrt{r}e)^{1/r}
16\frac{r}{e}n^2 \left(n+1 \right) .
\ea
Therefore, the steady state average queue length satisfies
\ba
  \frac{1}{\epsilon} \left(n-\frac{3}{2}+\frac{1}{2n}\right)-\widetilde{\Bi}(\epsilon,n) \leq \bE \left [ \sum\lij\bqij\se \right ] \leq \frac{1}{\epsilon} \left(n-\frac{3}{2}+\frac{1}{2n}\right)+\widetilde{\Bii}(\epsilon,n)
\ea
where
\ba
\widetilde{\Bi}(\epsilon,n) =&
\left(1-\frac{\epsilon}{2}\right)\left(n-2+\frac{1}{n}\right)
+n-\frac{1}{2}
+3n^{\left(2-\frac{1}{r}\right)}
\epsilon ^{\left(-\frac{1}{r}\right)} \widetilde{M}_r
\quad \text{and}\quad\\
\widetilde{\Bii}(\epsilon,n) =&
-\left(1-\frac{\epsilon}{2}\right)\left(n-2+\frac{1}{n}\right)
+\frac{n+1}{2}
+ 2n^{\left(2-\frac{1}{r}\right)}\epsilon^{\left(-\frac{1}{r}\right)}
\widetilde{M}_r
\ea
for any $r\in\{2,3,\ldots\}$. The terms $\widetilde{\Bi}(\epsilon,n)$ and $\widetilde{\Bii}(\epsilon,n)$
are both
 $o\left(\frac{1}{\epsilon}\right)$.
 In the heavy traffic limit as $\epsilon\downarrow 0$ which means as the mean arrival rate $\vlam^{\epsilon}\rightarrow \frac{1}{n}\vone$, we have
\ba
\lim_{\epsilon\downarrow 0}  \epsilon \bE \left [ \sum\lij\bqij\se \right ] =  \left(n-\frac{3}{2}+\frac{1}{2n}\right).
\ea
Thus, MaxWeight algorithm has optimal queue length scaling in the heavy traffic limit.
\end{theorem}
Thus, in the heavy traffic limit, we have a universal lower bound on the ($\epsilon$ scaled) average queue lengths that is $\Omega(n)$ and the MaxWeight policy achieves this bound within a factor less than $2$.
Since we are interested in the asymptotics  both in term of number of ports, $n$ and distance from boundary of the capacity region, $\epsilon$, there are several possible limits in which the system can be studied. Heavy traffic limit is one such asymptotic, where we first let the arrival rate approach the boundary of the capacity region and look at the scaling of average queue length in terms of $n$. Another set of asymptotic regimes is when $\epsilon \rightarrow 0$ and $n\rightarrow \infty$ simultaneously. This can be studied by setting $\epsilon=n^{-\beta}$ for $\beta>0$.
Such a limit was studied in \cite{shah2012optimal,zhong2014scaling} for scheduling algorithms that are different from the MaxWeight algorithms studied here.
The universal lower bound in such a limit is $\Omega(n^{(1+\beta)})$. It is now easy to see the following corollary.
\begin{corollary}
Consider a sequence switch systems with Bernoulli arrivals, indexed by $n$. The $n^{\text{th}}$ system has mean arrival rate vector $\vlam^{(n)}=\frac{1-\gamma_nn^{-\beta}}{n}\vone$ with $\beta>0$ and $\gamma_n>0$ is a sequence that is $\Theta(1)$. The load is $\rho^{(n)} = 1-\gamma_n n^{-\beta}$. Fix a scheduling policy under which the switch system is stable for any $n>0$.
Suppose that the queue lengths process $\vq^{(n)}(t)$ process converges in distribution to a steady state random vector $\bvq^{(n)}$.
Then, for each of these systems, the average queue length is lower bounded by
\ba
\bE \left [\sum\lij \bqij ^{(n)} \right ] \geq \frac{(1-\gamma_n n^{-\beta})^2}{2\gamma_n}n^\beta(n-1)
\ea
and so is $\Omega(n^{(1+\beta)})$.

Under the MaxWeight scheduling policy, the queue lengths process $\vq^{(n)}(t)$ process is positive recurrent and so converges to a steady state random vector in distribution $\bvq^{(n)}$.
When $2\gamma_n\leq n^{(\beta-1)}$,
the steady state average queue length satisfies
\ban
\frac{n^{(1+\beta)}}{\gamma_n}-\Biii(n)
  \leq \bE \left [ \sum\lij\bqij^{(n)} \right ] \leq
\frac{n^{(1+\beta)}}{\gamma_n}+\Biv(n) \label{eq:uniform_enc}
  \quad \text{for } \beta>4
\ean
where $\Biii(n)$ and $\Biv(n)$ are $o\left(n^{(1+\beta)}\right)$. Thus, under the MaxWeight algorithm, the average sum queue lengths is $\Theta(n^{(1+\beta)})$ and so has optimal scaling.
\end{corollary}
\begin{proof}
The universal lower bound directly follows from Theorem \ref{thm:uniform_arrivals} using $\epsilon^{(n)}=\gamma_n n^{-\beta}$. We will now prove the second part of the corollary which is under the MaxWeight policy.
Sine $2\gamma_n\leq n^{(\beta-1)}$, we have $0< \epsilon^{(n)} \leq 1/2n$
and Theorem \ref{thm:uniform_arrivals} is applicable. Therefore, we have (\ref{eq:uniform_enc}) with
\ba
\Biii(n) =&
\left(\frac{3n^\beta-n^{\beta-1}}{2\gamma_n}\right)
+\left(1\hspace{-.25em}-\hspace{-.25em}\frac{\gamma_n n^{-\beta}}{2}\right)\hspace{-.25em}\left(n-2+\frac{1}{n}\right)
+n-\frac{1}{2}
+48\left(\frac{2\sqrt{r}e}{\gamma_n}\right)^{1/r}
\frac{r}{e}
n^{\left(2-\frac{1}{r}+\frac{\beta}{r}\right)}
n^2 \left(n+1 \right),\\
\Biv(n) =&
\left(\frac{-3n^\beta+n^{\beta-1}}{2\gamma_n}\right)
-\left(1 \hspace{-.25em}-\hspace{-.25em} \frac{\gamma_nn^{-\beta}}{2}\right)\hspace{-.25em}\left(n-2+\frac{1}{n}\right)
+\frac{n+1}{2}
+ 32\left(\frac{2\sqrt{r}e}{\gamma_n}\right)^{1/r}
\frac{r}{e}n^{\left(2-\frac{1}{r}+\frac{\beta}{r}\right)}
n^2 \left(n+1 \right)
\ea
Clearly all but the last term above are $o\left(n^{(1+\beta)}\right)$. The last terms are $\Theta(n^{\left(5+\frac{\beta-1}{r}\right)})$. For any $\beta>4$, we can pick $r$ large enough so that $4+\frac{\beta-1}{r}<\beta$ and so we have that $\Biii(n)$ and $\Biv(n)$ are $o\left(n^{(1+\beta)}\right)$.
\end{proof}


\section{Conclusion \label{sec:conc}}

We have obtained a characterization of the heavy-traffic behavior of the sum queue length in steady-state in an $n\times n$ switch operating under the MaxWeight scheduling policy when all ports are saturated. We then considered the special case of uniform Bernoulli traffic and studied the switch in an asymptotic regime where the load increases simultaneously with the number of ports. We showed that the steady-state average queue lengths are within a factor less than $2$ of a universal lower bound.
The result settles one version of a conjecture regarding the performance of the MaxWeight policy. A number of extensions can be considered:
\begin{itemize}

\item 
    Extensions of the result to more general traffic patterns when only a few ports are saturated is an open problem.

\item We believe that one may be also be able to allow correlations across time slots by making an assumption similar to the assumption in Section II.C of \cite{ErySriPer_05}, and considering the drift of the Lyapunov function over multiple time slots. This extension may require a bit of additional work.

\item A Brownian limit has been established in the heavy-traffic regime in \cite{kang2009diffusion}, but a characterization of the behavior of this limit in steady-state is not known. We expect the mean of the sum queue lengths (multiplied by $\epsilon$ and in the limit $\epsilon\rightarrow 0$) in steady-state that we have derived to be equal to the sum of the steady-state expectations of the components of the Brownian motion in \cite{kang2009diffusion}. This would be interesting to verify.

\item Verifying whether the MaxWeight algorithm achieves the optimal queue-length scaling in the size of the switch in non-heavy-traffic regimes is still an open problem.

\end{itemize}

\section*{Acknowledgment} The work presented here was supported in part by NSF Grant ECCS-1202065.


\bibliographystyle{IEEEtran}
\bibliography{references}

\appendix
\section{Proof of Lemma \ref{lem:Bertsimas}} \label{sec:Bertsimas_proof}
\begin{proof}
Lemma \ref{lem:Hajek} is applicable here and so we have that $\bE[Z\left(\overline{X}\right)]<\infty$.
Recall that $\Delta Z(X)$ is a random variable for any $X$, so define
\ba
\widetilde{D} \triangleq \sup_{X\in \mathcal{X}} \text{ess\,sup} |\Delta Z(X)| = \sup_{X,X'\in\mathcal{X},  \bP(X(t+1)=X'|X(t)=X)>0} |Z(X')-Z(X)|
\ea
Also define
\ba
p_{\max} = \sup_{X\in \mathcal{X}} \bP(X(t+1)>X|X(t)=X)
\ea
Then, from Theorem 1 in \cite{bertsimas_momentbound}, we have
\ba
\bP\left(Z\left(\overline{X}\right)>\kappa+2\widetilde{D}m\right)\leq \left(\frac{\widetilde{D}p_{\max}}{\widetilde{D}p_{\max}+\eta} \right)^{m+1}.\ea
Clearly, $\widetilde{D} \leq D$ and $p_{max}\leq 1$. Therefore, we get
\ba
\bP\left(Z\left(\overline{X}\right)>\kappa+2Dm\right)\leq & \bP\left(Z\left(\overline{X}\right)>\kappa+2\widetilde{D}m\right) \\ \leq &\left(\frac{\widetilde{D}p_{\max}}{\widetilde{D}p_{\max}+\eta} \right)^{m+1}\\
\leq & \left(\frac{D}{D+\eta} \right)^{m+1},\ea
where the last inequality follows from $\widetilde{D}p_{\max} \leq D$ and $m+1 \geq 1$. This proves the first part of the lemma. We will now use this result to obtain moment bounds. Since $r>0$ and $Z(.)\geq0$, we have 
\ba
\bE[Z\left(\overline{X}\right)^r] =& r\int_{t=0}^{\infty} t^{r-1}\bP \left(Z\left(\overline{X}\right) >t\right)dt\\
=&r\int_{t=0}^{\kappa} t^{r-1}\bP \left(Z\left(\overline{X}\right) >t\right)dt +r\int_{t=\kappa}^{\infty} t^{r-1}\bP \left(Z\left(\overline{X}\right) >t\right)dt\\
\leq & r\int_{t=0}^{\kappa} t^{r-1}dt +r\sum_{m=0}^{\infty}\int_{t=\kappa+2Dm}^{\kappa+2D(m+1)} t^{r-1}\bP \left(Z\left(\overline{X}\right) >t\right)dt\\
\leq & \kappa^r +r\sum_{m=0}^{\infty}\int_{t=\kappa+2Dm}^{\kappa+2D(m+1)} t^{r-1}\bP \left(Z\left(\overline{X}\right) >\kappa+2Dm\right)dt\\
\leq & \kappa^r +\sum_{m=0}^{\infty} \left(\frac{D}{D+\eta} \right)^{m+1} \int_{t=\kappa+2Dm}^{\kappa+2D(m+1)} rt^{r-1}dt\\
= & \kappa^r +\sum_{m=0}^{\infty} \left(\frac{D}{D+\eta} \right)^{m+1} \left(\kappa+2D(m+1)\right)^r- \left(\kappa+2Dm\right)^r\\
= & \kappa^r\left(1-\frac{D}{D+\eta} \right) +\sum_{m=1}^{\infty} \left(\kappa+2Dm\right)^r \left[\left(\frac{D}{D+\eta} \right)^{m} -\left(\frac{D}{D+\eta} \right)^{m+1} \right]
\\
= & \left(\frac{\eta}{D+\eta} \right)  \left[
\sum_{m=0}^{\infty} \left(\kappa+2Dm\right)^r \left(\frac{D}{D+\eta} \right)^{m}
\right]\\
\stackrel{(a)}{\leq} & \left(\frac{\eta}{D+\eta} \right)  \left[
\sum_{m=0}^{\infty} \left(2\kappa\right)^r \left(\frac{D}{D+\eta} \right)^{m}
+\sum_{m=0}^{\infty} \left(4Dm\right)^r \left(\frac{D}{D+\eta} \right)^{m}
\right]\\
= &
\left(2\kappa\right)^r+
 \left(4D\right)^r \left(\frac{\eta}{D+\eta} \right)\sum_{m=0}^{\infty} m^r \left(\frac{D}{D+\eta} \right)^{m},
\ea
where (a) follows from $(a+b)^r\leq 2^r\max(a,b)^r\leq 2^r(a^r+b^r)$.
It is known \cite{comtet_combinatorics} that for $x<1$ and $r=1,2,\ldots $ 
$\sum_{m=0}^{\infty} m^r x^{m}= \frac{1}{(1-x)^{r+1}} \sum_{k=0}^{r-1} A(r,k) x^{k+1}$ where $A(r,k)$ are called the Eulerian numbers. It is also known that $\sum_{k=0}^{r-1} A(r,k)=r!$. Therefore, when $x<1$, we have $\sum_{m=0}^{\infty} m^r x^{m} \leq  \frac{1}{(1-x)^{r+1}} r!$. Using this relation, we get
\ba
\bE[Z\left(\overline{X}\right)^r]
\leq &
\left(2\kappa\right)^r+
 \left(4D\right)^r \left(\frac{D+\eta}{\eta} \right)^r r!.
\ea
\end{proof}

\section{Proof of Lemma \ref{lem:V_bounded}}\label{sec:V_bounded_proof}
\begin{proof}
We will use Lemma \ref{lem:Hajek} to first show that $\bE[\V(\bvq)]$ is finite.
Define the Lyapunov function $\Wq \triangleq \|\vq\| =\sqrt{\Vq}$, and its drift
\ba
\Delta \Wq \triangleq & [\W(\vq(t+1))-\W(\vq(t))]\>\mathcal{I}(\vq(t)=\vq)
\ea

We will first verify condition \ref{cond:C2} of Lemma \ref{lem:Hajek}. Using the same arguments as in ( \ref{eq:delta_Wperp}), we get
\ba
| \Delta \Wq| =& |\|\vq(t+1)\|-\|\vq(t)\||\>\mathcal{I}(\vq(t)=\vq) \nonumber\\
\leq & \|\vq(t+1)-\vq(t)\|\>\mathcal{I}(\vq(t)=\vq) \nonumber\\
\leq &n^2 \max\lij|\qij(t+1)-\qij(t)|\>\mathcal{I}(\vq(t)=\vq) \nonumber\\
\leq &n^2 a_{\max} ,
\ea
thus verifying condition \ref{cond:C2}. We will now verify condition \ref{cond:C1}
\ba
\bEq{\Delta\Wq} = &\bEq{\|\vq(t+1)\|-\|\vq(t)\|} \nonumber\\
=& \bEq{\sqrt{\|\vq(t+1)\|^2}-\sqrt{\|\vq(t)\|^2}}\\
\stackrel{(a)}{\leq}& \bEq{\frac{1}{2\|\vq(t)\|}\|\vq(t+1)\|^2-\|\vq(t)\|^2}\\
=&\frac{1}{2\|\vq\|}\bEq{\Delta\Vq}\\
\stackrel{(b)}{\leq} &  \frac{1}{2\|\vq\|} \left(
\|\vlam\|^2+\|\vsig\|^2+n +2\La\vq,\vlam-\bEq{\vs(t)}\Ra
\right) \\
\stackrel{(c)}{\leq} &  \frac{1}{2\|\vq\|} \left(
\|\vlam\|^2+\|\vsig\|^2+n +2\min_{\bf{r}\in \cC}\La\vq,\vlam-\bf{r}\Ra
\right) \\
\stackrel{(d)}{\leq} &  \frac{1}{2\|\vq\|} \left(
\|\vlam\|^2+\|\vsig\|^2+n +2\La\vq,\vlam-(\vlam+\epsilon_1\vone)\Ra
\right) \\
= &  \frac{\|\vlam\|^2+\|\vsig\|^2+n}{2\|\vq\|} - \epsilon_1 \frac{\|\vq\|_1}{\|\vq\|} \\
\stackrel{(e)}{\leq} &  \frac{\|\vlam\|^2+\|\vsig\|^2+n}{2\|\vq\|} - \epsilon_1 \\
\leq & - \frac{\epsilon_1}{2} \text{ for all } \vq \text{ such that } \Wq\geq \frac{\|\vlam\|^2+\|\vsig\|^2+n}{\epsilon_1}
\ea
where \vsig denotes the variance vector and  $\|\vq\|_1\triangleq \sum\lij\qij$ denotes the $\ell_1$ norm of $\vq$.
Inequality (a) follows from the concavity of square root function, due to which we have that $\sqrt{y}-\sqrt{x}\leq \frac{1}{2\sqrt{x}}(y-x)$. Inequality (b) follows from the bound on drift of $\V(.)$ obtained in (\ref{eq:delta_V}) in the proof of the proof of Proposition \ref{prop:SSC}; (c) follows from the fact that we use MaxWeight scheduling. Since $\vlam\in int(\cC)$, there exists a $\epsilon_1>0$ such that $\vlam+\epsilon_1\vone \in \cC$. This gives (d).
 For any vector $\vx$, its $\ell_1$ norm is at least its $\ell_2$ norm , i.e., $\|\vx\|_1\geq \|\vx\|$. This gives inequality (e). Thus, condition \ref{cond:C1} is verified and we have that all moments of $\W(\bvq)$ exist in steady state. In particular, we have that $\bE[\V(\bvq)]$ is finite.

Now, note that
\ba
\Vtq = \left(\sum\lij \qij \right)^2\leq \left(\sum\lij \max\lij \qij \right)^2 = n^4 \max\lij \qij^2 \leq n^4 \sum\lij \qij^2 = n^4 \Vq.
\ea
Thus, $\bE[\Vt(\bvq)]$ is also finite. The lemma follows by noting that $\Vi(\vq) \leq \Vt(\vq)$ and  $\Vo(\vq) \leq \Vt(\vq)$.
\end{proof}

\end{document}